\documentclass[12pt]{amsart}
\usepackage[alphabetic]{amsrefs}
\usepackage{mathdots, blkarray}
\usepackage{multicol}
\usepackage{graphicx}
\usepackage{caption}
\usepackage{amscd}
\usepackage{upgreek}
\usepackage{stmaryrd}
\usepackage{longtable}
\usepackage[T1]{fontenc}
\usepackage{latexsym, amsmath, amssymb, amsthm}
\usepackage[svgnames]{xcolor}
\usepackage{rsfso}
\usepackage[mathscr]{eucal}
\usepackage{mathtools}
\usepackage{mathptmx}
\usepackage{titletoc}
\usepackage{wrapfig}
\usepackage{float}
\usepackage{xypic}
\usepackage{microtype}
\usepackage{dsfont}
\usepackage{xcolor}
\usepackage{color}
\usepackage[colorlinks = true,
            linkcolor  = DarkBlue,
            urlcolor   = DarkRed,
            citecolor  = DarkGreen]{hyperref}
\usepackage{hyperref}
\allowdisplaybreaks	
\usepackage[centering, includeheadfoot, hmargin=1.0in, tmargin=0.5in,
  bmargin=0.6in, headheight=6pt]{geometry}
\newtheorem{theorem}{Theorem}[section]
\newtheorem{prop}[theorem]{Proposition}
\newtheorem{defn}[theorem]{\rm\textsc{Definition}}
\newtheorem{lem}[theorem]{Lemma}
\newtheorem{coro}[theorem]{Corollary}

\newtheorem{thm}[theorem]{Theorem}

\newtheorem{rem}[theorem]{\rm\textsc{Remark}}
\newtheorem{exam}[theorem]{\rm\textsc{Example}}

\newcommand{\bslash}{\kern-0.1em\texttt{\scalebox{0.6}[1]{/}}\kern-0.15em \texttt{\scalebox{0.6}[1]{/}}}

\begin{document}
\title[The logic of quasi-MV* algebras]{The logic of quasi-MV* algebras}

\author{Lei Cai}
\address{School of Mathematical Sciences, University of Jinan, No. 336, West Road of Nan Xinzhuang,
         Jinan, Shandong, 250022 P.R. China.}
\email{cailei@stu.ujn.edu.cn}

\author{Yingying Jiang}
\address{School of Mathematical Sciences, University of Jinan, No. 336, West Road of Nan Xinzhuang,
         Jinan, Shandong, 250022 P.R. China.}
\email{yyjiangmath@gmail.com}

\author{Wenjuan Chen}
\address{School of Mathematical Sciences, University of Jinan, No. 336, West Road of Nan Xinzhuang,
         Jinan, Shandong, 250022 P.R. China.}
\email{wjchenmath@gmail.com}

\begin{abstract}
Quasi-MV* algebras were introduced as generalizations of MV*-algebras and quasi-MV algebras in \cite{jc1}. The recent investigation into quasi-MV* algebras shows that they are closely related to quantum computational logic and complex fuzzy logic. In this paper, we aim to study the logical system associated with quasi-MV* algebras in detail. First, we introduce quasi-Wajsberg* algebras as the term equivalence of quasi-MV* algebras and investigate the related properties of quasi-Wajsberg* algebras. Then we establish the logical system associated with quasi-Wajsberg* algebras using fewer deduction rules. Finally, we discuss the soundness of this logical system.
\end{abstract}


\maketitle
\baselineskip=16.2pt

\dottedcontents{section}[1.16cm]{}{1.8em}{5pt}
\dottedcontents{subsection}[2.00cm]{}{2.7em}{5pt}
\section{Introduction}\label{intro}

The ${\L}$ukasiewicz infinite-valued propositional logic ${\L}$ is an important non-classical logical system with broad applications in fields such as fuzzy control, computer engineering, mathematical foundations, philosophical logic, and quantum computing. As is well-known, the truth values of the logic ${\L}$ are defined over the closed unit interval $[0,1]$. In \cite{c2}, Chang considered a natural extension ${\L}^{*}$ of ${\L}$, where the truth-value set is extended to the closed interval $[-1,1]$, and established several results for the logic ${\L}^{*}$. Notably, some results for ${\L}^{*}$ were direct analogues of known results for ${\L}$. To investigate the axiomatizability of certain sets, Chang further introduced and studied MV*-algebras in \cite{c2}. These algebras extend MV-algebras in a manner analogous to how the interval $[-1,1]$ naturally extends $[0,1]$. However, because the binary operation $+$ on an MV*-algebra is non-associative, it introduced complexities which were not present in the case for MV-algebra. Subsequently, Lewin et al. studied the algebraic structures of MV*-algebras in \cite{lsm1}, and discussed the soundness and completeness of ${\L}^{*}$ in \cite{ls,lsm2}.

In 2003, based on complex fuzzy sets, Ramot et al. proposed complex fuzzy logic as a generalization of traditional fuzzy logic \cite{rflk}. The set of truth values in complex fuzzy logic is the unit circle in the complex plane, i.e., $\{z\in \mathbb{C}| |z|\leq 1\}$. Since then, complex fuzzy logic has received increasingly attentions \cite{dyy,lz,yd}, and many authors have sought to extend the study of traditional fuzzy logic to complex fuzzy logic. For example, the algebraic product is a common operator in traditional fuzzy logic. In \cite{d05}, Dick applied the algebraic product as an operator in complex fuzzy logic and derived a partial ordering on the unit circle in the complex plane. However, the algebraic product in complex fuzzy logic does not retain the properties of its counterpart in traditional fuzzy logic. Dai later improved Dick's results in \cite{d21}. Indeed, existing research indicates that algebraic structures in traditional fuzzy logic cannot be directly transplanted to complex fuzzy logic.

Considering the geometric meaning of the unit circle in the complex plane, the set $\{z\in \mathbb{C}| |z|\leq 1\}$ consists of all points whose distance from the origin does not exceed 1. Following this view, the closed interval $[-1,1]$ can be viewed as the set of all points satisfying the distance to the origin is not more than 1 on the real axis. Thus $\{z\in \mathbb{C}| |z|\leq 1\}$ naturally generalizes $\{z\in \mathbb{R}| |z|\leq 1\}=[-1,1]$, implying that complex fuzzy logic can be regarded as a generalization of ${\L}^{*}$. Therefore, we believe that extending MV*-algebras may hold significant relevance for studying the algebraic structures of complex fuzzy logic.

On the other hand, to investigate complex fuzzy logic, Tamir et al. in \cite{tlk} proposed a Cartesian representation of the complex membership grade: $u(V,z)=u_{r}(V)+ju_{i}(z)$, where $u_{r}(V)$, $u_{i}(z)$ are in the unit interval $[0,1]$ and $\sqrt{-1}=j$. Based on this representation, they defined generalized complex propositional fuzzy logic using direct generalization of traditional fuzzy logic. Using Euler formula, the polar-to-cartesian coordinate transformation is expressed as $z=\rho e^{j\theta}=\rho \cos(\theta)+j\rho \sin(\theta)$, where $\rho \cos(\theta)$, $\rho \sin(\theta)$ are in the closed interval $[-1,1]$ and $\sqrt{-1}=j$. Consequently, the Cartesian product $[-1,1]\times [-1,1]$ serves as a more comprehensive domain for representing complex fuzzy logic.

Building on this framework, Jiang and Chen introduced quasi-MV* algebras as a generalization of MV*-algebras and quasi-MV algebras. In \cite{jc1}, the algebraic structure on $[-1,1]\times [-1,1]$ was shown to be a quasi-MV* algebra. Since the unit circle in the complex plane is embedded within $[-1,1]\times [-1,1]$, quasi-MV* algebras provide a unified algebraic framework to characterize the logical properties of complex fuzzy logic.
While previous studies \cite{jc2, jc3} have thoroughly characterized the ideals and filters of quasi-MV* algebras, this paper shifts focus to the logical systems associated with quasi-MV* algebras. Taking implication as a primitive connective better aligns with logical frameworks, therefore, quasi-MV* algebras are commonly replaced by their term equivalent variety. Building on this perspective, we first introduce quasi-Wajsberg* algebras in this paper.
The paper is organized as follows: In Section \ref{Pre}, we review some definitions and properties which will be used in what follows. In Section \ref{Sec-qw}, we give the definition of a quasi-Wajsberg* algebra and investigate the related properties of a quasi-Wajsberg* algebra. We also discuss the relationship between quasi-Wajsberg* algebras and quasi-MV* algebras. In Section \ref{Sec-L}, we establish the logical system associated with quasi-MV* algebras, define an equivalence relation which is based on semantics of the logical system, and show that the quotient algebra with respect to this equivalence relation is an MV*-algebra. Moreover, we also discuss the soundness of this logical system.

\section{Preliminary}\label{Pre}
In this section, we recall some definitions and results which will be used in the following.

\begin{defn}\cite{lsm1}
Let $\textbf{B}=\langle B;\oplus,{-},0,1\rangle$ be an algebra of type $\langle2,1,0,0\rangle$. If the following conditions are satisfied for any $x,y,z\in B$,

(MV*1)  $x \oplus y =y \oplus x$,

(MV*2)  $(1\oplus x) \oplus(y \oplus( 1\oplus z))=((1\oplus x) \oplus y) \oplus(1\oplus z)$,

(MV*3)  $x \oplus (-x)=0$,

(MV*4)  $(x \oplus 1)\oplus 1 =1$,

(MV*5)  $x \oplus 0 =x$,

(MV*6)  ${-}(x \oplus y) =(-x) \oplus (-y)$,

(MV*7)  ${-}(-x) =x$,

(MV*8)  $x \oplus y =(x^{+} \oplus y^{+} )\oplus (x^{-}\oplus y^{-})$,

(MV*9)  $(-x \oplus (x \oplus y))^{+} ={-}(x^{+})\oplus (x^{+} \oplus y^{+})$,

(MV*10)  $x \vee y =y \vee x$,

(MV*11)  $x \vee (y\vee z) =(x \vee y)\vee z$,

(MV*12)  $x \oplus (y\vee z) =(x \oplus y)\vee(x\oplus z)$,

\noindent in which ones define
$x^{+}=1 \oplus(-1 \oplus x)$,
$x^{-}=-1 \oplus(1 \oplus x)$, and
$x\vee y=(x^{+}\oplus(-x^{+}\oplus y^{+})^{+})\oplus(x^{-}\oplus(-x^{-}\oplus y^{-})^{+})$, then $\textbf{B}=\langle B;\oplus,{-},0,1\rangle$ is called an \emph{MV*-algebra}.
\end{defn}

\begin{exam}\cite{c2}\label{e0.0} Let $R=[-1,1]$ and operations be defined as follows. For any $a, b\in R$, $a \oplus_{\scriptscriptstyle \textbf{R}} b=\min{\{1, \max{\{-1, a+b\}}\}}$, $-_{\scriptscriptstyle \textbf{R}} a=-a$, $0_{\scriptscriptstyle \textbf{R}}=0$, and $1_{\scriptscriptstyle \textbf{R}}=1$.
Then $\textbf{R}=\langle R; \oplus_{\scriptscriptstyle \textbf{R}}, -_{\scriptscriptstyle \textbf{R}}, 0_{\scriptscriptstyle \textbf{R}}, 1_{\scriptscriptstyle \textbf{R}} \rangle$ is an MV*-algebra.
\end{exam}

\begin{defn}\cite{jc1}\label{d1.1}
Let $\textbf{A}=\langle A;\oplus,{-},^{+},^{-},0,1\rangle$ be an algebra of type $\langle2,1,1,1,0,0\rangle$. If the following conditions are satisfied for any ${x,y, z}\in A$,

(QMV*1) $x\oplus y=y\oplus x$,

(QMV*2) $(1\oplus x)\oplus(y\oplus(1\oplus z))=((1\oplus x)\oplus y)\oplus(1\oplus z)$,

(QMV*3) $(x\oplus 1)\oplus 1=1$,

(QMV*4) $(x\oplus y)\oplus 0=x\oplus y$,

(QMV*5) $x^{+}\oplus 0 = (x\oplus 0)^{+}=1 \oplus(-1 \oplus x)$,

\hspace{1.6cm} $x^{-}\oplus 0=(x\oplus 0)^{-} = -1 \oplus(1 \oplus x)$,

(QMV*6) $x\oplus y=(x^{+}\oplus y^{+})\oplus (x^{-}\oplus y^{-})$,

(QMV*7) $0=-0$,

(QMV*8) $x\oplus (-x)=0$,

(QMV*9) $-(x\oplus y)=-x\oplus (-y)$,

(QMV*10) $-(-x)=x$,

(QMV*11) $(-x\oplus(x\oplus y))^{+}=-x^{+}\oplus (x^{+}\oplus y^{+})$,

(QMV*12) $x\vee y=y\vee x$,

(QMV*13) $x\vee (y\vee z)=(x\vee y)\vee z$,

(QMV*14) $x\oplus (y\vee z)=(x\oplus y)\vee(x\oplus z)$,

\noindent in which $x\vee y = (x^{+}\oplus(-x^{+} \oplus y^{+})^{+})\oplus(x^{-}\oplus (-x^{-}\oplus y^{-})^{+})$, then $\textbf{A}=\langle A;\oplus,{-},^{+},^{-},0,1\rangle$ is called a \emph{quasi-MV* algebra}.
\end{defn}

\begin{exam}\label{e0.1}
Let $R^{*}=[-1,1]\times [-1,1]$ and operations be defined as follows. For any $\langle a,b\rangle, \langle c,d\rangle\in R^{*}$,

$\langle a, b\rangle \oplus_{\scriptscriptstyle\textbf{R}^{*}} \langle c, d\rangle=\langle \max{\{-1, \min{\{1, a+c\}}\}}, 0 \rangle$,

$-_{\scriptscriptstyle \textbf{R}^{*}}\langle a, b\rangle=\langle -a,-b\rangle$,

$\langle a, b\rangle^{+_{\textbf{R}^{*}}}=\langle \max\{0, a\}, \max\{0,b\} \rangle$,

$\langle a, b\rangle^{-_{\textbf{R}^{*}}}=\langle \min\{0, a\}, \min\{0,b\} \rangle$,

$0_{\scriptscriptstyle \textbf{R}^{*}}=\langle 0, 0\rangle$,

$1_{\scriptscriptstyle \textbf{R}^{*}}=\langle 1, 0\rangle$.

Then $\textbf{R}^{*}=\langle R^{*}; \oplus_{\scriptscriptstyle \textbf{R}^{*}}, -_{\scriptscriptstyle \textbf{R}^{*}}, ^{+_{\textbf{R}^{*}}}, ^{-_{\textbf{R}^{*}}}, 0_{\scriptscriptstyle \textbf{R}^{*}}, 1_{\scriptscriptstyle \textbf{R}^{*}} \rangle$ is a quasi-MV* algebra.
\end{exam}

\begin{exam}\cite{jc1}\label{e3.3}
Let $A=\{a,b,c,d,e,0,1\}$ be a 7-element set and operations be defined in Table 1 and Table 2. Then $\langle A;\oplus,-,^+,^-,0,1\rangle$ is a quasi-MV* algebra.

\begin{minipage}{\textwidth}
 \begin{minipage}[t]{0.45\textwidth}
  \centering
     \makeatletter\def\@captype{table}\makeatother\caption{the operation $\oplus$}
       \begin{tabular}{|c|c|c|c|c|c|c|c|}
\hline
$\oplus$ & $a$ &$b$&$c$&$0$&$d$&$e$&$1$ \\
\hline
$a$&$a$&$a$&$a$&$a$&$b$&$b$&$0$ \\
\hline
$b$&$a$&$a$&$a$&$b$&$0$&$0$&$e$ \\
\hline
$c$&$a$&$a$&$a$&$b$&$0$&$0$&$e$ \\
\hline
0&$a$&$b$&$b$&0&$e$&$e$&$1$ \\
\hline
$d$ &$b$&0&0&$e$&1&1&1 \\
\hline
$e$ &$b$&0&0&$e$&1&1&1 \\
\hline
1&0&$e$&$e$&1&1&1&1 \\
\hline
\end{tabular}
  \end{minipage}
  \begin{minipage}[t]{0.45\textwidth}
   \centering
        \makeatletter\def\@captype{table}\makeatother\caption{the operations $-, ^+, ^-$}
         \begin{tabular}{|c|c|c|c|c|c|c|c|}
\hline
{}&$a$&$b$&$c$&0&$d$&$e$&$1$\\
\hline
$-$&$1$&$e$&$d$&0&$c$&$b$&$a$\\
\hline
$^+$&$0$&$0$&$0$&0&$d$&$e$&$1$\\
\hline
$^-$&$a$&$b$&$c$&0&$0$&$0$&$0$\\
\hline
\end{tabular}
   \end{minipage}
\end{minipage}
\end{exam}

\begin{lem}\emph{\cite{jc1}}\label{l0.1}
Let $\emph{\textbf{A}}=\langle A;\oplus,{-},^{+},^{-},0,1\rangle$ be a quasi-MV* algebra. Then for any $x,y\in A$, we have

\emph{(1)} $0\oplus 0=0$, $1\oplus 0=1$, $-1\oplus 0=-1$, $1\oplus 1=1$ and $-1\oplus (-1)=-1$,

\emph{(2)} $-(x\oplus 0)=-x\oplus 0$,

\emph{(3)} $0^{+}=0=0^{-}$, $1^{+}=1$, $1^{-}=0$, $(-1)^{+}=0$ and $(-1)^{-}=-1$,

\emph{(4)}  $(-x)^{+}\oplus 0=-x^{-}\oplus 0$ and $(-x)^{-}\oplus 0=-x^{+}\oplus 0$,

\emph{(5)} $x^{-+}\oplus 0=0=x^{+-}\oplus 0$,

\emph{(6)} $x^{++}\oplus 0=x^{+}\oplus 0$ and $x^{--}\oplus 0=x^{-}\oplus 0$,

\emph{(7)} $x\vee 0=x^{+}\oplus 0$ and $x\wedge 0=x^{-}\oplus 0$,

\emph{(8)} $x^{+}\vee 0=x^{+}\oplus 0$ and $x^{-}\wedge 0=x^{-}\oplus 0$,

\emph{(9)} $x^{+}\wedge 0=0$ and $x^{-}\vee 0=0$,

\emph{(10)} $x\vee x=x\oplus 0$ and $x\wedge x=x\oplus 0$,

\emph{(11)} $x\oplus y=(x\oplus 0)\oplus y=x\oplus (y\oplus 0)=(x\oplus 0)\oplus (y\oplus 0)$,

\emph{(12)} $x\vee y=(x\vee y)\oplus 0=(x\oplus 0)\vee y=x\vee (y\oplus 0)$ and

 \hspace{0.85cm}$x\wedge y=(x\wedge y)\oplus 0=(x\oplus 0)\wedge y=x\wedge (y\oplus 0)$,

\emph{(13)} $x\vee y=(x^{+}\vee y^{+})\oplus(x^{-}\vee y^{-})$ and $x\wedge y=(x^{+}\wedge y^{+})\oplus(x^{-}\wedge y^{-})$,

\emph{(14)} $x^{+}\vee x^{-}=x^{+}\oplus 0$ and $x^{+}\wedge x^{-}=x^{-}\oplus 0$,

\emph{(15)} $x\oplus 0=(x\oplus 0)^{+}\oplus (x\oplus 0)^{-}=(x^{+}\vee x^{-})\oplus(x^{+}\wedge x^{-})=x^+\oplus x^-$.

%
%
\end{lem}

\begin{lem}\emph{\cite{jc1}}\label{l0.3}
Let $\emph{\textbf{A}}=\langle A;\oplus,{-},^{+},^{-},0,1\rangle$ be a quasi-MV* algebra. Then for any $x,y,u,v\in A$, we have

\emph{(1)} If $x\leq y$ and $y\leq x$, then $x\oplus 0=y\oplus 0$,

\emph{(2)} $-1\leq x\leq 1$,

\emph{(3)} $x\leq x\oplus 0$ and $x\oplus 0\leq x$,

\emph{(4)} If $x\le y$ and $u\le v$, then $x\oplus u\le y\oplus v$,

\emph{(5)} If $x\leq y$, then $x^{+}\leq y^{+}$ and $x^{-}\leq y^{-}$,

\emph{(6)} If $x\le y$ and $u\le v$, then $x\vee u\le y\vee v$,

\emph{(7)} If $x\le y$ and $u\le v$, then $x\wedge u\le y\wedge v$,

\emph{(8)} $x\leq y$ iff $x\wedge y=x\oplus 0$,

\emph{(9)} If $x\leq y$, then $-y\leq -x$,

\emph{(10)} If $x\oplus 0=1\oplus z$ for some $z\in A$, then $0\leq x$,

\emph{(11)} If $x\oplus 0=-1\oplus z$ for some $z\in A$, then $x\leq 0$,

\emph{(12)} $x\leq 0$ iff $x^{+}\oplus 0=0$, and $x\geq 0$ iff $x^{-}\oplus 0=0$,

\emph{(13)} $x^-\le x\le x^+$.
\end{lem}

In \cite{c2}, the corresponding logical system $\L^{*}$ of MV*-algebras was introduced by Chang. Later, the logical system was investigated further in \cite{ls,lsm2}. Below we list the axioms and rules of $\L^*$.

\noindent\textbf{Axioms schemas}

(P1) $(p\to q)\leftrightarrow (\neg q\to \neg p)$,

(P2) $p\leftrightarrow ((q\to q)\to p)$,

(P3) $\neg(p\to q)\leftrightarrow (q\to p)$,

(P4) $p\to 1$,

(P5) $1\leftrightarrow ((1\to p)\to 1)$,

(P6) $((p\to 1)\to((q\to 1)\to r))\to ((q\to 1)\to((p\to 1)\to r))$,

(P7) $(p\to q)\leftrightarrow((q^{+}\to p^{-})\to (p^{+}\to q^{-}))$,

(P8) $(p\to (\neg p\to q))^+\leftrightarrow (p^+\to (\neg p^+\to q^+))$,

(P9) $(p\to (q\vee r))\leftrightarrow((p\to r)\vee (p\to q))$,

(P10) $(p\vee (q\vee r))\leftrightarrow ((p\vee q)\vee r)$,

\noindent in which $p^{+}=(p\to 1)\to 1$, $p^{-}=(p\to \neg 1)\to \neg 1$ and $p\vee q = ((p^{+}\to q^{+})^{+})\to(\neg p)^{-} )\to((q^{-}\to p^{-})^{-}\to p^{-})$. Moreover, $p\leftrightarrow q$ means that $p\to q$ and $q\to p$.

\noindent\textbf{Rules of deduction}

(M1) $p, p\to q\vdash_{\tiny\L^{*}} q$,

(M2) $p\to q, r\to t\vdash_{\tiny\L^{*}} (q\to r)\to (p\to t)$,

(M3) $p\vdash_{\tiny\L^{*}} p^-$.

\section{Quasi-Wajsberg* algebras}\label{Sec-qw}

In this section, we introduce the notion of a quasi-Wajsberg* algebra and study the related properties. We also show that quasi-Wajsberg* algebras and quasi-MV* algebras are term equivalence.

\begin{defn}\label{dqw}
Let $\textbf{W}=\langle W;\to,\neg,^{+},^{-},1\rangle$ be an algebra of type $\langle2,1,1,1,0\rangle$. If the following conditions are satisfied for any ${x,y,z}\in W$,

(QW*1) $x\to y=\neg y\to \neg x$,

(QW*2) $(x\to 1)\to((y\to 1)\to z)=(y\to 1)\to((x\to 1)\to z)$,

(QW*3) $(1\to x)\to 1=1$,

(QW*4) $(z\to z)\to (x\to y)=x\to y$,

(QW*5) $(1\to 1)\to x^{+}=((1\to 1)\to x)^{+}=(x\to 1)\to 1$ and

\hspace{1.4cm} $(1\to 1)\to x^{-}=((1\to 1)\to x)^{-}=(x\to \neg 1)\to \neg 1$,

(QW*6) $x\to y=(y^{+}\to x^{-})\to (x^{+}\to y^{-})$,

(QW*7) $\neg(x\to y)= y\to x$,

(QW*8) $\neg\neg x=x$,

(QW*9) $(x\to (\neg x\to y))^{+}=x^{+}\to (\neg x^{+}\to y^{+})$,

(QW*10) $x\vee y=y\vee x$,

(QW*11) $x\vee (y\vee z)=(x\vee y)\vee z$,

(QW*12) $x\to (y\vee z)=(x\to y)\vee (x\to z)$,

\noindent in which $x\vee y = ((x^{+}\to y^{+})^{+}\to(\neg x)^{-}) \to((y^{-}\to x^{-})^{-}\to x^{-})$, then $\textbf{W}=\langle W;\to,\neg,^{+},^{-},1\rangle$ is called a \emph{quasi-Wajsberg* algebra}.
\end{defn}

\begin{exam}\label{e0.2}
Let $R^{*}=[-1,1]\times [-1,1]$ and operations be defined as follows. For any $\langle a, b\rangle, \langle c, d\rangle\in R^{*}$,

$\langle a, b\rangle\to \langle c, d\rangle=\langle \min\{1, \max\{-1,c-a\}\}, 0\rangle$,

$\neg \langle a, b\rangle=\langle -a, -b\rangle$,

$\langle a, b\rangle^{+}=\langle \max\{0, a\}, b\rangle$,

$\langle a, b\rangle^{-}=\langle \min\{0, a\}, b\rangle$,

$1_{\scriptscriptstyle \mathbf{R}^*}=\langle 1, 0\rangle$.

Then $\mathbf{R}^*= \langle R^{*}; \to, \neg, ^{+}, ^{-}, 1_{\scriptscriptstyle \mathbf{R}^*} \rangle$ is a quasi-Wajsberg* algebra.
\end{exam}

In the following, we abbreviate the quasi-Wajsberg* algebra $\textbf{W}=\langle W;\to,\neg,^{+},^{-},1\rangle$ as $\textbf{W}$. For any quasi-Wajsberg* algebra, we consider that the operations $^+$ and $^-$ (which have the same priority) have priority to the operations $\to$ and $\neg$, the operation $\neg$ has priority to the operation $\to$.

\begin{rem}\label{n0.1}
For a quasi-Wajsberg* algebra $\mathbf{W}$ and any $x,y\in W$, we define an operation $x\wedge y=\neg(\neg x\vee \neg y)$, then the operation $\wedge$ satisfies commutative law and associative law by (QW*10), (QW*8) and (QW*11). Below, the operation $\to$ has priority to the operations $\vee$ and $\wedge$.
\end{rem}

\begin{prop}\label{l0.2}
Let $\emph{\textbf{W}}$ be a quasi-Wajsberg* algebra. Then for any $x,y\in W$, we have

\emph{(1)} $\neg x\to y=\neg y\to x$ and $x\to \neg y=y\to\neg x$,

\emph{(2)} $\neg (x\to y)=\neg x\to \neg y$,

\emph{(3)} $x\to x=y\to y$,

\emph{(4)} $\neg(x\to x)=x\to x$.
\end{prop}

\begin{proof}
(1) We have $\neg x\to y=\neg x\to \neg\neg y=\neg y\to x$ and $x\to \neg y=\neg\neg x\to \neg y=y\to\neg x$ by (QW*8) and (QW*1).

(2) We have $\neg(x\to y)=y\to x=\neg\neg y\to \neg\neg x=\neg x\to \neg y$ by (QW*7), (QW*8) and (QW*1).

(3) By (QW*4) and (QW*7), we have $y\to y=(x\to x)\to (y\to y)=\neg((y\to y)\to(x\to x))=\neg(x\to x)=x\to x$.

(4) We have $\neg (x\to x)=\neg x \to \neg x=x\to x$ by (2) and (3).
\end{proof}

\begin{rem}\label{r1.1}
Let $\textbf{W}$ be a quasi-Wajsberg* algebra. We define $0=1\to 1$ in $\textbf{W}$. From Proposition \ref{l0.2}(3) and (4), we have $0=x\to x$ for any $x\in W$ and $0=\neg 0$. Therefore, (QW*5) can be written as $0\to x^+ = (0\to x)^+ = (x\to 1)\to 1$ and $0\to x^- = (0\to x)^- = (x\to \neg1)\to \neg1$.
\end{rem}

\begin{prop}\label{pr.1}
Let $\emph{\textbf{W}}$ be a quasi-Wajsberg* algebra. Then we have

\emph{(1)} $0\to 1=1$ and $0\to \neg1=\neg 1$,

\emph{(2)} $1\to 0=\neg1 $ and $\neg1 \to 0=1$,

\emph{(3)} $\neg1\to 1=1$ and $1\to \neg1=\neg1$,

\emph{(4)} $0^+ = 0=0^-$, $1^+ =1$, $(\neg1)^- = \neg1$.

\emph{(5)} $1^- = 0=(\neg1)^+$.
\end{prop}

\begin{proof}
(1) Since $0=1\to 1$, we have $0\to 1 = (1\to 1)\to 1=1$ by (QW*3). Moreover, we have $0\to \neg 1=\neg 0 \to \neg 1=\neg (0\to 1)=\neg 1$ from Remark \ref{r1.1} and Proposition \ref{l0.2}(2).

(2) We have $1\to 0 = \neg (0\to 1)=\neg 1$ from (QW*7) and (1). Moreover, we have $\neg 1\to 0 = \neg 0 \to 1 = 0\to 1 =1$ from Proposition \ref{l0.2}(1), Remark \ref{r1.1} and (1).

(3) We have $\neg1\to 1= \neg(0\to 1) \to 1= (1\to 0)\to 1=1$ from (1), (QW*7) and (QW*3). Moreover, we have $1\to \neg1 = \neg \neg 1\to \neg1 = \neg (\neg1\to 1)=\neg 1$ from (QW*8) and Proposition \ref{l0.2}(2).

(4) Since $0=0\to 0$, we have $0^+ =(0\to 0)^+ = (0\to 1)\to 1=1\to 1=0$ and $0^- =(0\to 0)^- = (0\to \neg1)\to \neg1 =\neg 1 \to \neg 1=0$ from (QW*5) and (1). Moreover, we have $1^+ = (0\to 1)^+ = (1\to 1)\to 1=0\to 1=1$ and $(\neg1)^- = (0\to \neg1)^- = (\neg1 \to \neg1)\to \neg1 = 0\to \neg1 =\neg1$ from (1) and (QW*5).

(5) We have $1^- = (0\to 1)^- = (1\to \neg1)\to \neg1 = \neg1 \to \neg1 = 0$ and $(\neg1)^+ = (0\to \neg1)^+ = (\neg1 \to 1)\to 1 =1\to 1=0$ from (1), (QW*5), (3) and Remark \ref{r1.1}.
\end{proof}

\begin{lem}\label{le.1}
Let $\emph{\textbf{W}}$ be a quasi-Wajsberg* algebra. Then for any $x,y \in W$, we have

\emph{(1)} $0\to (x\to y)=x\to y$ and $0\to \neg(x\to y)=\neg(x\to y)$,

\emph{(2)} $(x\to y)\to 0 =\neg(x\to y)$ and $\neg(x\to y)\to 0=x\to y$.
\end{lem}

\begin{proof}
(1) From (QW*4) and Remark \ref{r1.1}, we can get $0\to (x\to y)=x\to y$ immediately. Moreover, we have $0\to \neg(x\to y)= 0\to (\neg x\to \neg y)= \neg x\to \neg y=\neg(x\to y)$ from Proposition \ref{l0.2}(2).

(2) We have $(x\to y)\to 0 = \neg0 \to \neg(x\to y) = 0 \to \neg(x\to y)= \neg(x\to y)$ from (QW*1), Remark \ref{r1.1} and (1). Moreover, we have $\neg(x\to y)\to 0= \neg 0\to \neg \neg(x\to y) = 0\to (x\to y)=x\to y$ from (QW*1), Remark \ref{r1.1}, (QW*8) and (1).
\end{proof}

\begin{prop}\label{pr.3}
Let $\emph{\textbf{W}}$ be a quasi-Wajsberg* algebra. Then for any $x,y,z \in W$, we have

\emph{(1)} $x\vee x=0\to x$ and $x\wedge x=0\to x$,

\emph{(2)} $x\to (y\wedge z)=(x\to y)\wedge (x\to z)$, $(x\wedge y)\to z=(x\to z)\vee (y\to z)$ and $(x\vee y)\to z=(x\to z)\wedge (y\to z)$,

\emph{(3)} $x\to y=(0\to x)\to y=x\to(0\to y)=(0\to x)\to(0\to y)$,

\emph{(4)} $x\vee y=0\to (x\vee y)=(0\to x)\vee y=x\vee(0\to y)=(0\to x)\vee (0\to y)$ and $x\wedge y=0\to (x\wedge y)=(0\to x)\wedge y=x\wedge(0\to y)=(0\to x)\wedge (0\to y)$.
\end{prop}

\begin{proof}
(1) We have
\begin{align*}
x\vee x &=((x^+\to x^+)^+\to (\neg x)^-)\to ((x^-\to x^-)^-\to x^-)  \\
&=(0^+\to (\neg x)^-)\to (0^-\to x^-) &&\text{(by Remark \ref{r1.1})} \\
&=(0\to (\neg x)^-)\to (0\to x^-) &&\text{(by Proposition \ref{pr.1}(4))} \\
&=((\neg x \to \neg1)\to \neg1)\to (0\to x^-) &&\text{(by Remark \ref{r1.1})}\\
&=\neg ((x\to 1)\to 1)\to (0\to x^-) &&\text{(by Proposition \ref{l0.2}(2))}\\
&=\neg (0\to x^+)\to (0\to x^-) &&\text{(by Remark \ref{r1.1})} \\
&=(x^+\to 0)\to (0\to x^-) &&\text{(by (QW*7))} \\
&=(x^+\to 0^-)\to (0^+\to x^-) &&\text{(by Proposition \ref{pr.1}(4))} \\
&=0\to x &&\text{(by (QW*6))}.
\end{align*}
Based on this result, we can get
\begin{align*}
x\wedge x &=\neg(\neg x\vee \neg x) &&\text{(by Remark \ref{n0.1})}\\
&=\neg(0\to \neg x) &&\text{(by (1))}\\
&=\neg 0\to \neg\neg x &&\text{(by Proposition \ref{l0.2}(2))}\\
&=0\to x &&\text{(by Remark \ref{r1.1} and (QW*8))}.
\end{align*}

(2) We have
\begin{align*}
x\to (y\wedge z)&=x\to \neg(\neg y\vee \neg z) &&\text{(by Remark \ref{n0.1})}\\
&=\neg\neg x\to \neg(\neg y\vee \neg z) &&\text{(by (QW*8))}\\
&=\neg (\neg x\to (\neg y\vee \neg z))  &&\text{(by Proposition \ref{l0.2}(2))}\\
&=\neg ((\neg x\to \neg y)\vee (\neg x\to \neg z))  &&\text{(by (QW*12))}\\
&=\neg (\neg(x\to y)\vee \neg(x\to z))  &&\text{(by Proposition \ref{l0.2}(2))}\\
&=(x\to y)\wedge (x\to z) &&\text{(by Remark \ref{n0.1})},
\end{align*}
similarly, we can get that
\begin{align*}
(x\wedge y)\to z &=\neg(\neg x\vee \neg y) \to z &&\text{(by Remark \ref{n0.1})}\\
&=\neg z\to \neg(\neg(\neg x\vee \neg y)) &&\text{(by (QW*1))}\\
&=\neg z\to (\neg x\vee \neg y) &&\text{(by (QW*8))}\\
&=(\neg z\to \neg x)\vee(\neg z\to \neg y) &&\text{(by (QW*12))}\\
&=(x\to z)\vee(y\to z) &&\text{(by (QW*1))}
\end{align*}
and
\begin{align*}
(x\vee y)\to z &=\neg z\to \neg(x\vee y) &&\text{(by (QW*1))}\\
&=\neg z\to \neg(\neg (\neg x) \vee \neg (\neg y)) &&\text{(by (QW*8))}\\
&=\neg z\to (\neg x\wedge \neg y) &&\text{(by Remark \ref{n0.1})}\\
&=(\neg z\to \neg x)\wedge (\neg z\to \neg y) &&\text{(by Remark \ref{n0.1})}\\
&=(x\to z)\wedge (y\to z) &&\text{(by (QW*1))}.
\end{align*}

(3) We have
\begin{align*}
(0\to x)\to y &=(x\vee x)\to y &&\text{(by (1))}\\
&=(x\to y)\wedge (x\to y) &&\text{(by (2))}\\
&=0\to (x\to y) &&\text{(by (1))}\\
&=x\to y &&\text{(by Lemma \ref{le.1}(1))}.
\end{align*}
The others can be proved similarly.

(4) We have
\begin{align*}
x\vee y &=0\to (x\vee y) &&\text{(by Lemma \ref{le.1}(1))}\\
&=(0\to x)\vee(0\to y) &&\text{(by (QW*12))}\\
&=(0\to (0\to x))\vee(0\to y) &&\text{(by (3))}\\
&=0\to ((0\to x)\vee y) &&\text{(by (QW*12))}\\
&=(0\to x)\vee y &&\text{(by Lemma \ref{le.1}(1))}.
\end{align*}
The rest can be obtained similarly.
\end{proof}

\begin{prop}\label{pr.2}
Let $\emph{\textbf{W}}$ be a quasi-Wajsberg* algebra. Then for any $x,y \in W$, we have

\emph{(1)} $0\to (\neg x)^+=0\to \neg x^-$ and $0\to (\neg x)^-=0\to \neg x^+$,

\emph{(2)} $0\to x^{+-}= 0 =0\to x^{-+}$,

\emph{(3)} $0\to x^{++}=0\to x^+$ and $0\to x^{--}=0\to x^-$,

\emph{(4)} $(\neg x)^+ \to y=\neg x^- \to y$, $x\to (\neg y)^+ = x\to (\neg y^-)$, $(\neg x)^- \to y=\neg x^+ \to y$ and $x\to (\neg y)^- = x\to (\neg y^+)$,

\emph{(5)} $x^{+-}\to y=0\to y=x^{-+}\to y$ and $x\to y^{+-}= x\to 0=x\to y^{-+}$,

\emph{(6)} $x^{++}\to y = x^+ \to y$, $x\to y^{++}=x\to y^+$, $x^{--}\to y = x^- \to y$ and $x\to y^{--}=x\to y^-$.
\end{prop}

\begin{proof}
(1) We have
\begin{align*}
0\to (\neg x)^+ &=(\neg x\to 1)\to 1 &&\text{(by Remark \ref{r1.1})}\\
&=(\neg x\to \neg\neg 1)\to \neg\neg 1 &&\text{(by (QW*8))}\\
&=\neg((x\to \neg 1)\to \neg 1) &&\text{(by Proposition \ref{l0.2}(2))}\\
&=\neg(0\to x^-) &&\text{(by Remark \ref{r1.1})}\\
&=\neg 0\to \neg x^- &&\text{(by Proposition \ref{l0.2}(2))}\\
&=0\to \neg x^- &&\text{(by Remark \ref{r1.1})}.
\end{align*}
Moreover,
\begin{align*}
0\to (\neg x)^- &=(\neg x\to \neg 1)\to \neg 1 &&\text{(by Remark \ref{r1.1})}\\
&=\neg ((x\to 1)\to 1) &&\text{(by Proposition \ref{l0.2}(2))}\\
&=\neg(0\to x^+) &&\text{(by Remark \ref{r1.1})}\\
&=\neg0\to \neg x^+ &&\text{(by Proposition \ref{l0.2}(2))}\\
&=0\to \neg x^+ &&\text{(by Remark \ref{r1.1})}.
\end{align*}

(2) We have
\begin{align*}
0\to x^{+-}&=(0\to x^+)^- &&\text{(by Remark \ref{r1.1})}\\
&=((x\to 1)\to 1)^- &&\text{(by Remark \ref{r1.1})}\\
&=(0\to((x\to 1)\to 1))^- &&\text{(by Lemma \ref{le.1}(1))}\\
&=(((x\to 1)\to 1)\to \neg 1)\to \neg 1 &&\text{(by Remark \ref{r1.1})}\\
&=((\neg 1\to \neg(x\to 1))\to \neg 1)\to \neg 1 &&\text{(by (QW*1))}\\
&=\neg((1\to (x\to 1))\to 1)\to \neg 1 &&\text{(by Proposition \ref{l0.2}(2))}\\
&=\neg 1\to \neg 1 &&\text{(by (QW*3))}\\
&=0 &&\text{(by Remark \ref{r1.1})}.
\end{align*}
Moreover,
\begin{align*}
0\to x^{-+}&=(0\to x^-)^+ &&\text{(by Remark \ref{r1.1})}\\
&=((x\to \neg 1)\to \neg 1)^+ &&\text{(by Remark \ref{r1.1})}\\
&=(0\to((x\to \neg 1)\to \neg 1))^+ &&\text{(by Lemma \ref{le.1}(1))}\\
&=(((x\to \neg 1)\to \neg 1)\to 1)\to 1 &&\text{(by Remark \ref{r1.1})}\\
&=((1\to \neg (x\to \neg 1))\to 1)\to 1 &&\text{(by Proposition \ref{l0.2}(1))}\\
&=1\to 1 &&\text{(by (QW*3))}\\
&=0 &&\text{(by Remark \ref{r1.1})}.
\end{align*}

(3) We have
\begin{align*}
0\to x^{++}&=(0\to x^+)^+ &&\text{(by Remark \ref{r1.1})}\\
&=((x\to 1)\to 1)^+ &&\text{(by Remark \ref{r1.1})}\\
&=(0\to ((x\to 1)\to 1))^+ &&\text{(by Lemma \ref{le.1}(1))}\\
&=(((x\to 1)\to 1)\to 1)\to 1 &&\text{(by Remark \ref{r1.1})}\\
&=\neg(1\to (((x\to 1)\to 1)\to 1)) &&\text{(by (QW*7))}\\
&=\neg((\neg 1\to 1)\to (((x\to 1)\to 1)\to 1)) &&\text{(by Proposition \ref{pr.1}(3))}\\
&=\neg(((x\to 1)\to 1)\to ((\neg 1\to 1)\to 1)) &&\text{(by (QW*2))}\\
&=\neg (((x\to 1)\to 1)\to (1\to 1)) &&\text{(by Proposition \ref{pr.1}(3))}\\
&= \neg (((x\to 1)\to 1)\to 0) &&\text{(by Remark \ref{r1.1})}\\
&=\neg\neg((x\to 1)\to 1) &&\text{(by Lemma \ref{le.1}(2))}\\
&=(x\to 1)\to 1 &&\text{(by (QW*8))}\\
&=0\to x^+ &&\text{(by Remark \ref{r1.1})}.
\end{align*}
Similarly, we have
\begin{align*}
0\to x^{--}&=(0\to x^-)^- &&\text{(by Remark \ref{r1.1})}\\
&=((x\to \neg 1)\to \neg 1)^- &&\text{(by Remark \ref{r1.1})}\\
&=(0\to ((x\to \neg 1)\to \neg 1))^- &&\text{(by Lemma \ref{le.1}(1))}\\
&=(((x\to \neg 1)\to \neg 1)\to \neg 1)\to \neg 1 &&\text{(by Remark \ref{r1.1})}\\
&=(((\neg\neg x\to \neg 1)\to \neg 1)\to \neg 1)\to \neg 1 &&\text{(by (QW*8))}\\
&=\neg((((\neg x\to 1)\to 1)\to 1)\to 1) &&\text{(by Proposition \ref{l0.2}(2))}\\
&=\neg(0\to (\neg x)^+) &&\text{(by Remark \ref{r1.1})}\\
&=\neg(0\to \neg x^-) &&\text{(by (1))}\\
&=\neg 0\to \neg\neg x^- &&\text{(by Proposition \ref{l0.2}(2))}\\
&=0\to x^- &&\text{(by Remark \ref{r1.1} and (QW*8))}.
\end{align*}

(4) From Proposition \ref{pr.3}(3) and (1), we have $(\neg x)^+ \to y=(0\to (\neg x)^+)\to y=(0\to \neg x^-)\to y = \neg x^- \to y$.
The rest can be proved similarly.

(5) From Proposition \ref{pr.3}(3) and (2), we have $x^{+-} \to y = (0\to x^{+-})\to y = 0\to y$ and $x^{-+} \to y = (0\to x^{-+})\to y = 0\to y$. Similarly, we can get that $x\to y^{+-}= x\to 0=x\to y^{-+}$.

(6) From Proposition \ref{pr.3}(3) and (3), we have $x^{++}\to y = (0\to x^{++})\to y= (0\to x^+)\to y = x^+\to y$. The rest can be proved similarly.
\end{proof}

\begin{prop}\label{pr.4}
Let $\emph{\textbf{W}}$ be a quasi-Wajsberg* algebra. Then for any $x,y,z\in W$, we have

\emph{(1)} $x\vee0=0\to x^+$ and $x\wedge 0=0\to x^-$,

\emph{(2)} $x^+\vee 0=0\to x^{+}$, $x^-\wedge 0=0\to x^{-}$, $x^-\vee 0=0$ and $x^+\wedge 0=0$.

\emph{(3)} $x\vee y=\neg(x^+\vee y^+)\to (x^-\vee y^-)$ and $x\wedge y=\neg(x^+\wedge y^+)\to (x^-\wedge y^-)$,

\emph{(4)} $x^+\vee x^-=0\to x^+$ and $x^+\wedge x^-=0\to x^-$,

\emph{(5)} $0 \to x =\neg x^+\to x^-$,

\emph{(6)} $(x\vee y)^+=x^+\vee y^+$, $(x\vee y)^-=x^-\vee y^-$, $(x\wedge y)^+=x^+\wedge y^+$ and $(x\wedge y)^-=x^-\wedge y^-$,

\emph{(7)} $x\vee(x\wedge y)=x\vee x$ and $x\wedge(x\vee y)=x\wedge x$.
\end{prop}

\begin{proof}
(1) Since $x\vee 0 =((x^+\to 0^+)^+\to (\neg x)^-)\to ((0^-\to x^-)^-\to x^-)$, we calculate $(x^+\to 0^+)^+\to (\neg x)^-$ and $(0^-\to x^-)^-\to x^-$ respectively. Because
\begin{align*}
(x^+\to 0^+)^+\to (\neg x)^- &= (x^+\to 0)^+\to (\neg x)^- &&\text{(by Proposition \ref{pr.1}(4))}\\
&=(\neg(0 \to x^+))^+\to (\neg x)^- &&\text{(by (QW*7))}\\
&=\neg(0\to x^+)^- \to \neg x^+ &&\text{(by Proposition \ref{pr.2}(4))}\\
&=\neg(0\to x^{+-})\to \neg x^+ &&\text{(by Remark \ref{r1.1})}\\
&= \neg0 \to \neg x^+ && \text{(by Proposition \ref{pr.2}(2))}\\
&= 0\to \neg x^+ && \text{(by Remark \ref{r1.1})}
\end{align*}
and
\begin{align*}
(0^-\to x^-)^-\to x^- &=(0\to x^-)^-\to x^- &&\text{(by Proposition \ref{pr.1}(4))}\\
&= (0\to x^{--})\to x^- && \text{(by Remark \ref{r1.1})}\\
&=(0\to x^-)\to x^- && \text{(by Proposition \ref{pr.2}(3))}\\
&=x^-\to x^- &&\text{(by Proposition \ref{pr.3}(3))}\\
&= 0 &&\text{(by Remark \ref{r1.1})},
\end{align*}
then
\begin{align*}
x\vee 0 &=(0\to \neg x^+)\to 0  \\
&=\neg (0\to \neg x^+) &&\text{(by Lemma \ref{le.1}(2))}\\
&= \neg0 \to \neg \neg x^+ && \text{(by Proposition \ref{l0.2}(2))} \\
&=0\to x^+  &&\text{(by Remark \ref{r1.1} and (QW*8))}.
\end{align*}
Based on this result, we have
\begin{align*}
x\wedge 0&=\neg(\neg x\vee \neg0) &&\text{(by Remark \ref{n0.1})}\\
&=\neg(\neg x\vee 0) &&\text{(by Remark \ref{r1.1})}\\
&=\neg(0\to (\neg x)^+) &&\text{(by (1))}\\
&=\neg(0\to\neg x^-) &&\text{(by Proposition \ref{pr.2}(1))}\\
&=\neg0\to \neg\neg x^- && \text{(by Proposition \ref{l0.2}(2))}\\
&=0\to x^- &&\text{(by Remark \ref{r1.1} and (QW*8))}.
\end{align*}

(2) From (1) and Proposition \ref{pr.2}(3), we have $x^+ \vee 0=0\to x^{++}=0\to x^+$ and $x^- \wedge 0 = 0\to x^{--} = 0\to x^-$. Moreover, from (1) and Proposition \ref{pr.2}(2), we have $x^- \vee 0=0\to x^{-+}=0$ and $x^+ \wedge 0=0\to x^{+-} = 0$.

(3) We have
\begin{align*}
x^+\vee y^+&=((x^{++}\to y^{++})^+\to (\neg x^+)^-)\to((y^{+-}\to x^{+-})^-\to x^{+-})\\
&=((x^{+}\to y^{+})^+\to ((\neg x)^{--}))\to 0 &&\text{(by Proposition \ref{pr.2}(4), (5), (6))}\\
&=((x^{+}\to y^{+})^+\to (\neg x)^{-})\to 0 &&\text{(by Proposition \ref{pr.2}(6))}\\
&=\neg((x^{+}\to y^{+})^+\to (\neg x)^{-}) &&\text{(by Lemma \ref{le.1}(2))}
\end{align*}
and
\begin{align*}
x^-\vee y^-&=((x^{-+}\to y^{-+})^+\to (\neg x^-)^-)\to((y^{--}\to x^{--})^-\to x^{--})\\
&=(0^+\to (\neg x)^{+-})\to((y^{-}\to x^{-})^-\to x^{-}) &&\text{(by Proposition \ref{pr.2}(4), (5), (6))}\\
&=(0\to 0)\to((y^{-}\to x^{-})^-\to x^{-}) &&\text{(by Proposition \ref{pr.2}(5))}\\
&=0\to((y^{-}\to x^{-})^-\to x^{-}) &&\text{(by Remark \ref{r1.1})}\\
&=(y^{-}\to x^{-})^-\to x^{-} &&\text{(by Lemma \ref{le.1}(1))}.
\end{align*}
It is clearly that $x\vee y =\neg(x^+\vee y^+)\to (x^-\vee y^-)$.
From the former result, we have
\begin{align*}
x\wedge y&=\neg(\neg x\vee \neg y)  &&\text{(by Remark \ref{n0.1})}\\
&=\neg(\neg((\neg x)^+\vee (\neg y)^+)\to ((\neg x)^-\vee (\neg y)^-)) &&\text{(by (3))} \\
&=\neg(\neg(\neg x^-\vee \neg y^-)\to (\neg x^+\vee \neg y^+)) &&\text{(by Proposition \ref{pr.2}(4))} \\
&=(\neg x^-\vee \neg y^-)\to \neg(\neg x^+\vee \neg y^+)  &&\text{(by Proposition \ref{l0.2}(2))} \\
&=\neg(x^-\wedge y^-)\to (x^+\wedge y^+) &&\text{(by Remark \ref{n0.1})} \\
&=\neg(x^+\wedge y^+)\to (x^-\wedge y^-) &&\text{(by Proposition \ref{l0.2}(1))}.
\end{align*}

(4) We have
\begin{align*}
x^+\vee x^- &=\neg(x^{++}\vee x^{-+})\to(x^{+-}\vee x^{--}) &&\text{(by (3))}\\
&=\neg(x^+\vee 0)\to (0 \vee x^-) &&\text{(by Proposition \ref{pr.2}(5), (6))} \\
&=\neg(0\to x^+)\to 0 &&\text{(by (2))} \\
&=0\to x^+ &&\text{(by Lemma \ref{le.1}(2))},
\end{align*}
and
\begin{align*}
x^+\wedge x^- &=\neg(x^{++}\wedge x^{-+})\to(x^{+-}\wedge x^{--}) &&\text{(by (3))}\\
&=\neg(x^+\wedge 0)\to (0\wedge x^-) &&\text{(by Proposition \ref{pr.2}(5), (6))}\\
&=\neg0\to (0\to x^-) &&\text{(by (2))}\\
&=0\to (0\to x^-) &&\text{(by Remark \ref{r1.1})}\\
&=0\to x^- &&\text{(by Lemma \ref{le.1}(1))}.
\end{align*}

(5) We have
\begin{align*}
0\to x&=(x^+\to 0^-)\to(0^+\to x^-) &&\text{(by (QW*6))}\\
&=(x^+\to 0)\to(0\to x^-) &&\text{(by Proposition \ref{pr.1}(4))}\\
&=\neg (0\to x^+)\to (0\to x^-) &&\text{(by (QW*7))}\\
&=(\neg 0 \to \neg x^+) \to (0\to x^-) &&\text{(by Proposition \ref{l0.2}(2))}\\
&= (0\to \neg x^+)\to (0\to x^-) &&\text{(by Remark \ref{r1.1})}\\
&=\neg x^+\to x^- &&\text{(by Proposition \ref{pr.3}(3))}.
\end{align*}

(6) We have
\begin{align*}
(x\vee y)^+ &=(0\to (x\vee y))^+ &&\text{(by Proposition \ref{pr.3}(4))}\\
&=((x\vee y)\to 1)\to 1 &&\text{(by Remark \ref{r1.1})}\\
&=((x\to 1)\wedge (y\to 1))\to 1 &&\text{(by Proposition \ref{pr.3}(2))}\\
&=((x\to 1)\to 1)\vee((y\to 1)\to 1) &&\text{(by Proposition \ref{pr.3}(2))}\\
&=(0\to x^+)\vee(0\to y^+) &&\text{(by Remark \ref{r1.1})}\\
&=x^+\vee y^+ &&\text{(by Proposition \ref{pr.3}(4))}.
\end{align*}
The others can be proved similarly.

(7) Since $x\vee(x\wedge y)=((x^+\to(x\wedge y)^+)^+\to(\neg x)^-)\to(((x\wedge y)^-\to x^-)^-\to x^-)$, we calculate $(x^+\to(x\wedge y)^+)^+\to(\neg x)^-$ and $((x\wedge y)^-\to x^-)^-\to x^-$ respectively. From (6), Proposition \ref{pr.3}(2), Remark \ref{r1.1}, Proposition \ref{pr.2}(2) and Proposition \ref{pr.2}(4), we have
\begin{align*}
(x^+\to(x\wedge y)^+)^+\to(\neg x)^- &= (x^+\to(x^+\wedge y^+))^+\to(\neg x)^- && \text{(by (6))}\\
&=((x^+\to x^+)\wedge(x^+\to y^+))^+\to (\neg x)^- &&\text{(by Proposition \ref{pr.3}(2))}\\
&=(0\wedge(x^+\to y^+))^+\to (\neg x)^- &&\text{(by Remark \ref{r1.1})}\\
&=(0\to(x^+\to y^+)^-)^+\to (\neg x)^- &&\text{(by (1))}\\
&=(0\to(x^+\to y^+)^{-+})\to (\neg x)^- &&\text{(by Remark \ref{r1.1})}\\
&=0\to \neg x^+ &&\text{(by Proposition \ref{pr.2}(2), (4))}
\end{align*}
and
\begin{align*}
((x\wedge y)^-\to x^-)^-\to x^- &=((x^-\wedge y^-)\to x^-)^-\to x^- && \text{(by (6))}\\
&=((x^-\to x^-)\vee(y^-\to x^-))^-\to x^- &&\text{(by Proposition \ref{pr.3}(2))}\\
&=(0\vee(y^-\to x^-))^-\to x^- &&\text{(by Remark \ref{r1.1})}\\
&= (0\to(y^-\to x^-)^+)^-\to x^- &&\text{(by (1))}\\
&=(0\to(y^-\to x^-)^{+-})\to x^- &&\text{(by Remark \ref{r1.1})}\\
&=0\to x^- &&\text{(by Proposition \ref{pr.2}(2))}.
\end{align*}
It turns out that
\begin{align*}
x\vee(x\wedge y) &=(0\to \neg x^+)\to (0\to x^-) \\
&= \neg x^+\to x^- &&\text{(by Proposition \ref{pr.3}(3))}\\
&= 0\to x &&\text{(by (5))}\\
&= x\vee x &&\text{(by Proposition \ref{pr.3}(1))}.
\end{align*}
Moreover, from the former result, we have
\begin{align*}
x\wedge (x\vee y)&=\neg(\neg x\vee \neg(x\vee y)) &&\text{(by Remark \ref{n0.1})}\\
&=\neg(\neg x\vee (\neg x\wedge \neg y)) &&\text{(by Remark \ref{n0.1}) and (QW*8)}\\
&=\neg(0\to \neg x) &&\text{(by (7))}\\
&=\neg 0\to \neg\neg x &&\text{(by Proposition \ref{l0.2}(2))}\\
&=0\to x &&\text{(by Remark \ref{r1.1} and (QW*8))}\\
&=x\wedge x &&\text{(by Proposition \ref{pr.3}(1))}.
\end{align*}
\end{proof}

\begin{coro}
Let $\mathbf{W}$ be a quasi-Wajsberg* algebra. Then $\langle W;\vee,\wedge \rangle$ is a quasi-lattice.
\end{coro}

\begin{proof}
We only show that $x\vee (y\vee y)=x\vee y$ and $x\wedge (y\wedge y)=x\wedge y$ for any $x,y\in W$. The rest can be obtained from Definition \ref{dqw} and Proposition \ref{pr.3}.
From Proposition \ref{pr.3}(1) and (4), we have $x\vee (y\vee y)= x \vee (0\to y) = x\vee y$ and $x\wedge (y\wedge y)=x \wedge (0\to y) = x \wedge y$, then $x\vee (y\vee y)=x\vee y$ and $x\wedge (y\wedge y)=x\wedge y$ for any $x,y\in W$.
\end{proof}

Let $\textbf{W}$ be a quasi-Wajsberg* algebra. For any $x,y\in W$, we define a relation $x\le y$ iff $x\vee y=0\to y$. If $x\le y$, we also say $y\ge x$. For any $x\in W$, if $0\le x$, then we call the element $x$ \emph{non-negative}, and if $x\le 0$, then we call the element $x$ \emph{non-positive}.
Moreover, if $0\le x$ and $0\neq 0\to x$, then we call the element $x$ \emph{positive} and denote it by $0<x$, if $x\le 0$ and $0 \neq 0 \to x$, then we call the element $x$ \emph{negative} and denote it by $x< 0$.

In the following, we denote $W^{\scriptscriptstyle\leq 0}=\{x\in W: x\leq 0\}$ and $W^{\scriptscriptstyle\geq 0}=\{x\in W: 0\leq x\}$. For any $x\in W$, we have $x^-\le 0\le x^+$ by Proposition \ref{pr.4}(2). Denote $W^+=\{x^+:x\in W\}$ and $W^-=\{x^-:x\in W\}$. Then we have $0\le y$ for any $y\in W^+$ and $z\le 0$ for any $z\in W^-$, so $W^+\subseteq W^{\scriptscriptstyle\geq 0}$ and $W^-\subseteq W^{\scriptscriptstyle\leq 0}$.

\begin{prop}\label{p3.2}
Let $\mathbf{A}=\langle A;\oplus,{-},^{+},^{-},0,1\rangle$ be a quasi-MV* algebra. If for any $x,y\in A$, we define $x\to y=-x\oplus y$ and $\neg x=-x$, then $f(\mathbf{A})=\langle A;\to,\neg,^{+},^{-},1\rangle$ is a quasi-Wajsberg* algebra, where $x\vee_{\scriptscriptstyle f(\mathbf{A})} y = ((x^{+}\to y^{+})^{+}\to(\neg x)^{-}) \to((y^{-}\to x^{-})^{-}\to x^{-})$.
\end{prop}

\begin{proof}
For any $x,y,z\in A$, we need to check the conditions (QW*1)--(QW*12) one by one.

(QW*1) We have
\begin{align*}
x\to y&=-x\oplus y &&\text{(by Definition)}\\
&=y\oplus (-x) &&\text{(by (QMV*1))}\\
&=-(-y)\oplus (-x) &&\text{(by (QMV*10))}\\
&=\neg y\to \neg x &&\text{(by Definition)},
\end{align*}

(QW*2) From (QMV*10), (QMV*9), (QMV*1) and (QMV*2), we have
\begin{align*}
(x\to 1)\to((y\to 1)\to z)&=-(-x\oplus 1)\oplus (-(-y\oplus 1)\oplus z) &&\text{(by Definition)}\\
&=-((-x\oplus 1)\oplus ((-y\oplus 1)\oplus (-z))) &&\text{(by (QMV*9) and (QMV*10))}\\
&=-((1\oplus (-x))\oplus ((1\oplus (-y))\oplus (-z))) &&\text{(by (QMV*1))}\\
&=-((1\oplus (-x))\oplus (-z\oplus(1\oplus (-y)))) &&\text{(by (QMV*1))}\\
&=-(((1\oplus (-x))\oplus (-z))\oplus(1\oplus (-y))) &&\text{(by (QMV*2))}\\
&=-(-y\oplus 1)\oplus (-(-x\oplus 1)\oplus z) &&\text{(by (QMV*1))}\\
&=(y\to 1)\to ((x\to 1)\to z) &&\text{(by Definition)}.
\end{align*}

(QW*3) We have
\begin{align*}
(1\to x)\to 1&=-(-1\oplus x)\oplus 1 &&\text{(by Definition)}\\
&=(1\oplus (-x))\oplus 1 &&\text{(by (QMV*9))}\\
&=(-x\oplus 1)\oplus 1 &&\text{(by (QMV*1))}\\
&=1 &&\text{(by (QMV*3))}.
\end{align*}

(QW*4) We have
\begin{align*}
(z\to z)\to (x\to y)&=-(-z\oplus z)\oplus (-x\oplus y) &&\text{(by Definition)}\\
&=0\oplus (-x\oplus y) &&\text{(by (QMV*8))}\\
&=-x\oplus y &&\text{(by (QMV*4))}\\
&=x\to y &&\text{(by Definition)}.
\end{align*}

(QW*5) From (QMV*8), (QMV*7) and (QMV*5), we have $(1\to 1)\to x^{+}=-(-1\oplus 1)\oplus x^+=0\oplus x^+$ and $((1\to 1)\to x)^{+}=(-(-1\oplus 1)\oplus x)^{+}=(0\oplus x)^{+}=0\oplus x^{+}$. Moreover, we can get that $(x\to 1)\to 1=-(-x\oplus 1)\oplus 1=(x\oplus (-1))\oplus 1=1\oplus(-1\oplus x)=0\oplus x^{+}$ from (QMV*9) and (QMV*5). It turns out that $(1\to 1)\to x^{+}=((1\to 1)\to x)^{+}=(x\to 1)\to 1$.
Similarly, we have $(1\to 1)\to x^{-}=((1\to 1)\to x)^{-}=(x\to \neg 1)\to \neg 1$.

(QW*6) We have
\begin{align*}
x\to y&=-x\oplus y &&\text{(by Definition)}\\
&=((-x)^{+}\oplus y^{+})\oplus ((-x)^{-}\oplus y^{-}) &&\text{(by (QMV*6))}\\
&=(((-x)^{+}\oplus 0)\oplus y^{+})\oplus (((-x)^{-}\oplus 0)\oplus y^{-}) &&\text{(by Lemma \ref{l0.1}(11))}\\
&=((-x^{-}\oplus 0)\oplus y^{+})\oplus ((-x^{+}\oplus 0)\oplus y^{-}) &&\text{(by Lemma \ref{l0.1}(4))}\\
&=(-x^{-}\oplus y^{+})\oplus (-x^{+}\oplus y^{-}) &&\text{(by Lemma \ref{l0.1}(11))}\\
&=-(-y^{+}\oplus x^{-})\oplus (-x^{+}\oplus y^{-}) &&\text{(by (QMV*9) and (QMV*1))}\\
&=(y^{+}\to x^{-})\to (x^{+}\to y^{-}) &&\text{(by Definition)}.
\end{align*}

(QW*7) We have $\neg(x\to y)=-(-x\oplus y)=-(y\oplus (-x))=-y\oplus x= y\to x$ by (QMV*1), (QMV*9) and (QMV*10).

(QW*8) We have $\neg\neg x=-(-x)=x$ by (QMV*10).

(QW*9) From (QMV*10) and (QMV*11), we have
\begin{align*}
(x\to (\neg x\to y))^{+}&=(-x\oplus (-(-x)\oplus y))^{+} &&\text{(by Definition)}\\
&=(-x\oplus (x\oplus y))^{+} &&\text{(by (QMV*10))}\\
&=-x^{+}\oplus (x^{+}\oplus y^{+}) &&\text{(by (QMV*11))}\\
&=-x^{+}\oplus (-(-x^{+})\oplus y^{+}) &&\text{(by (QMV*10))}\\
&=x^{+}\to (\neg x^{+}\to y^{+}) &&\text{(by Definition)}.
\end{align*}

(QW*10) We have
\begin{align*}
x\vee_{\scriptscriptstyle f(\mathbf{A})} y&=((x^{+}\to y^{+})^{+}\to(\neg x)^{-}) \to((y^{-}\to x^{-})^{-}\to x^{-}) &&\text{(by Definition)}\\
&=-(-(-x^{+}\oplus y^{+})^{+}\oplus (-x)^{-})\oplus (-((-y^{-})\oplus x^{-})^{-}\oplus x^{-}) &&\text{(by Definition)}\\
&=-(-(-x^{+}\oplus y^{+})^{+}\oplus ((-x)^{-}\oplus 0))\oplus (-((-y^{-})\oplus x^{-})^{-}\oplus x^{-}) &&\text{(by Lemma \ref{l0.1}(11))}\\
&=-(-(-x^{+}\oplus y^{+})^{+}\oplus (-x^{+}\oplus 0))\oplus (-((-y^{-})\oplus x^{-})^{-}\oplus x^{-}) &&\text{(by Lemma \ref{l0.1}(4))}\\
&=-(-(-x^{+}\oplus y^{+})^{+}\oplus (-x^{+}))\oplus (-(x^{-}\oplus (-y^{-}))^{-}\oplus x^{-}) &&\text{(by Lemma \ref{l0.1}(11))}\\
&=-(-((-x^{+}\oplus y^{+})^{+}\oplus x^{+}))\oplus (-(-(-x^{-}\oplus y^{-})^{+})\oplus x^{-}) &&\text{(by (QMV*9))}\\
&=((-x^{+}\oplus y^{+})^{+}\oplus x^{+})\oplus ((-x^{-}\oplus y^{-})^{+}\oplus x^{-}) &&\text{(by (QMV*10))}\\
&=(x^{+}\oplus(-x^{+}\oplus y^{+})^{+})\oplus (x^{-}\oplus(-x^{-}\oplus y^{-})^{+}) &&\text{(by (QMV*1))}\\
&=x\vee_{\scriptscriptstyle \mathbf{A}} y,
\end{align*}
so $x\vee_{\scriptscriptstyle f(\mathbf{A})} y=x\vee_{\scriptscriptstyle \mathbf{A}} y=y\vee_{\scriptscriptstyle \mathbf{A}} x=y\vee_{\scriptscriptstyle f(\mathbf{A})} x$
by (QMV*12).

(QW*11) Based on the proof of (QW*10), we have $x\vee_{\scriptscriptstyle f(\mathbf{A})} (y\vee_{\scriptscriptstyle f(\mathbf{A})} z)=(x\vee_{\scriptscriptstyle f(\mathbf{A})} y)\vee_{\scriptscriptstyle f(\mathbf{A})} z$ by (QMV*13).

(QW*12) We have $x\to (y\vee_{\scriptscriptstyle f(\mathbf{A})} z)=-x\oplus (y\vee_{\scriptscriptstyle \mathbf{A}} z)=(-x\oplus y)\vee_{\scriptscriptstyle \mathbf{A}} (-x\oplus z)=(x\to y)\vee_{\scriptscriptstyle f(\mathbf{A})} (x\to z)$ by (QMV*14).

Hence $f(\textbf{A})=\langle A;\to,\neg,^{+},^{-},1\rangle$ is a quasi-Wajsberg* algebra.
\end{proof}

According to Proposition \ref{p3.2},  given a quasi-MV* algebra $\mathbf{A}=\langle A;\oplus,{-},^{+},^{-},0,1\rangle$, we define $x\to y=-x\oplus y$, then $\langle A;\to,-,^{+},^{-},1\rangle$ is a quasi-Wajsberg* algebra. Hence we have the following example.

\begin{exam}
Let $\textbf{R}^{*}=\langle R^{*}; \oplus_{\scriptscriptstyle \textbf{R}^{*}}, -_{\scriptscriptstyle \textbf{R}^{*}}, ^{+_{\textbf{R}^{*}}}, ^{-_{\textbf{R}^{*}}}, 0_{\scriptscriptstyle \textbf{R}^{*}}, 1_{\scriptscriptstyle \textbf{R}^{*}} \rangle$ be a quasi-MV* algebra defined in Example 2.2.
We define $x\to y=-x \oplus_{\scriptscriptstyle \textbf{R}^{*}} y$, then $\langle R^{*}; \to, -_{\scriptscriptstyle \textbf{R}^{*}}, ^{+_{\textbf{R}^{*}}}, ^{-_{\textbf{R}^{*}}},  1_{\scriptscriptstyle \textbf{R}^{*}} \rangle$ is a quasi-Wajsberg* algebra which is different from Example 3.1.
\end{exam}

\begin{prop}\label{p3.3}
Let $\mathbf{W}=\langle W;\to,\neg,^{+},^{-},1\rangle$ be a quasi-Wajsberg* algebra. If for any $x,y\in W$, we define $0=x\to x$, $x\oplus y=\neg x\to y$ and $-x=\neg x$, then $g(\mathbf{W})=\langle W;\oplus,{-},^{+},^{-},0,1\rangle$ is a quasi-MV* algebra, where $x\vee_{\scriptscriptstyle g(\mathbf{W})} y = (x^{+}\oplus(-x^{+} \oplus y^{+})^{+})\oplus(x^{-}\oplus (-x^{-}\oplus y^{-})^{+})$.
\end{prop}

\begin{proof} For any $x,y,z\in A$, we prove that the conditions (QMV*1)--(QMV*14) hold.

(QMV*1) We have $x\oplus y=\neg x\to y=\neg y\to x=y\oplus x$ by Proposition \ref{l0.2}(1).

(QMV*2) From (QW*8), Proposition \ref{l0.2}(2), (1), (QW*7) and (QW*2), we have
\begin{align*}
(1\oplus x)\oplus (y\oplus (1\oplus z))&=\neg(\neg1\to x)\to (\neg y\to (\neg 1\to z)) &&\text{(by Definition)} \\
&=\neg(\neg1\to x)\to \neg\neg(\neg y\to (\neg 1\to z)) &&\text{(by (QW*8))}\\
&=\neg((\neg 1\to x)\to \neg(\neg y\to (\neg 1\to z))) &&\text{(by Proposition \ref{l0.2}(2))}\\
&=\neg((\neg x\to 1)\to \neg(\neg y\to (\neg z\to 1))) &&\text{(by Proposition \ref{l0.2}(1))}\\
&=\neg((\neg x\to 1)\to ((\neg z\to 1)\to \neg y)) &&\text{(by (QW*7))}\\
&=\neg((\neg z\to 1)\to ((\neg x\to 1)\to \neg y)) &&\text{(by (QW*2))}\\
&=\neg((\neg 1\to z)\to ((\neg 1\to x)\to \neg y)) &&\text{(by Proposition \ref{l0.2}(1))}\\
&=\neg((\neg 1\to z)\to (\neg\neg(\neg 1\to x)\to \neg y)) &&\text{(by (QW*8))}\\
&=\neg((\neg 1\to z)\to \neg(\neg(\neg 1\to x)\to y)) &&\text{(by Proposition \ref{l0.2}(2))}\\
&=\neg(\neg(\neg 1\to x)\to y)\to (\neg 1\to z) &&\text{(by (QW*7))}\\
&=((1\oplus x)\oplus y)\oplus(1\oplus z) &&\text{(by Definition)}.
\end{align*}

(QMV*3) We have $(x\oplus 1)\oplus 1=\neg(\neg x\to 1)\to 1=(1\to \neg x)\to 1=1$ by (QW*3).

(QMV*4) We have
\begin{align*}
(x\oplus y)\oplus 0&=\neg(\neg x\to y)\to 0 &&\text{(by Definition)}\\
&=\neg0\to (\neg x\to y) &&\text{(by Proposition \ref{l0.2}(1))}\\
&=0\to (\neg x\to y) &&\text{(by Remark \ref{r1.1})}\\
&=\neg x\to y &&\text{(by Lemma \ref{le.1}(1))}\\
&=x\oplus y &&\text{(by Definition)}.
\end{align*}

(QMV*5) We have
\begin{align*}
x^{+}\oplus 0&=\neg x^{+}\to 0 &&\text{(by Definition)}\\
&=\neg 0\to x^{+} &&\text{(by Proposition \ref{l0.2}(1))}\\
&=0\to x^{+} &&\text{(by Remark \ref{r1.1})}\\
&=(0\to x)^{+} &&\text{(by Remark \ref{r1.1})}\\
&=(\neg0\to x)^{+} &&\text{(by Remark \ref{r1.1})}\\
&=(0\oplus x)^{+} &&\text{(by Definition)}\\
&=(x\oplus 0)^{+} &&\text{(by (QMV*1))}
\end{align*}
and
\begin{align*}
x^{+}\oplus 0&=0\to x^{+} \\
&=(x\to 1)\to 1 &&\text{(by Remark \ref{r1.1})}\\
&=\neg 1\to \neg(x\to 1) &&\text{(by (QW*1))}\\
&=\neg 1\to (1\to x) &&\text{(by (QW*7))}\\
&=\neg 1\to (\neg \neg 1\to x) &&\text{(by (QW*8))}\\
&=1\oplus(-1\oplus x) &&\text{(by Definition)},
\end{align*}
 so $x^{+}\oplus 0=(x\oplus 0)^{+}=1\oplus(-1\oplus x)$. Similarly, we can get $x^{-}\oplus 0=(x\oplus 0)^{-}=-1\oplus(1\oplus x)$.

(QMV*6) From (QW*6), Proposition \ref{pr.3}(3), Proposition \ref{pr.2}(1) and (QMV*7), we have
\begin{align*}
x\oplus y&=\neg x\to y &&\text{(by Definition)}\\
&=(y^{+}\to (\neg x)^{-})\to ((\neg x)^{+}\to y^{-}) &&\text{(by (QW*6))}\\
&=(y^{+}\to \neg x^{+})\to (\neg x^{-}\to y^{-}) &&\text{(by Proposition \ref{pr.2}(4))}\\
&=\neg(\neg x^{+}\to y^{+})\to (\neg x^{-}\to y^{-}) &&\text{(by (QW*7))}\\
&=(x^{+}\oplus y^{+})\oplus(x^{-}\oplus y^{-}) &&\text{(by Definition)}.
\end{align*}

(QMV*7) From Remark \ref{r1.1}, we have $0=\neg 0 = -0$.

(QMV*8) We have $x\oplus (-x)=\neg x\to \neg x= 0$ by Remark \ref{r1.1}.

(QMV*9) By Proposition \ref{l0.2}(2), we have $-(x\oplus y)=\neg(\neg x\to y)=\neg \neg x\to \neg y=-x\oplus (-y)$.

(QMV*10) We have $-(-x)=\neg \neg x=x$ by (QW*8).

(QMV*11) We have
\begin{align*}
(-x\oplus(x\oplus y))^{+}&=(\neg \neg x\to (\neg x\to y))^{+} &&\text{(by Definition)}\\
&=(x\to (\neg x\to y))^{+} &&\text{(by (QW*8))}\\
&=x^{+}\to (\neg x^{+}\to y^{+}) &&\text{(by (QW*9))}\\
&=\neg\neg x^{+}\to (\neg x^{+}\to y^{+}) &&\text{(by (QW*8))}\\
&=-x^{+}\oplus(x^{+}\oplus y^{+}) &&\text{(by Definition)}.
\end{align*}

(QMV*12) Using (QW*8), (QW*7), Proposition \ref{l0.2}(1),  and Proposition \ref{pr.2}(1)), we have
\begin{align*}
x\vee_{\scriptscriptstyle g(\mathbf{W})} y&=(x^{+}\oplus(-x^{+} \oplus y^{+})^{+})\oplus(x^{-}\oplus (-x^{-}\oplus y^{-})^{+}) &&\text{(by Definition)}\\
&=\neg(\neg x^{+}\to (\neg \neg x^{+}\to y^{+})^{+})\to(\neg x^{-}\to(\neg \neg x^{-}\to y^{-})^{+}) &&\text{(by Definition)} \\
&=((x^{+}\to y^{+})^{+}\to \neg x^{+})\to (\neg(x^{-}\to y^{-})^{+}\to x^{-}) &&\text{(by (QW*7) and (QW*8))}\\
&=((x^{+}\to y^{+})^{+}\to (\neg x)^{-})\to ((\neg(x^{-}\to y^{-}))^{-}\to x^{-}) &&\text{(by Proposition \ref{pr.2}(4))}\\
&=((x^{+}\to y^{+})^{+}\to (\neg x)^{-})\to ((y^{-}\to x^{-})^{-}\to x^{-}) &&\text{(by (QW*7))}\\
&=x\vee_{\scriptscriptstyle \mathbf{W}} y,
\end{align*}
 so $x\vee_{\scriptscriptstyle g(\mathbf{W})} y=x\vee_{\scriptscriptstyle \mathbf{W}} y=y\vee_{\scriptscriptstyle \mathbf{W}} x=y\vee_{\scriptscriptstyle g(\mathbf{W})} x$ by (QW*10).

(QMV*13) Based on the proof of (QMV*12), we have $x\vee_{\scriptscriptstyle g(\mathbf{W})}(y\vee_{\scriptscriptstyle g(\mathbf{W})} z)=(x\vee_{\scriptscriptstyle g(\mathbf{W})} y)\vee_{\scriptscriptstyle g(\mathbf{W})} z$ by (QW*11).

(QMV*14) We have $x\oplus (y\vee_{\scriptscriptstyle g(\mathbf{W})} z)=\neg x\to (y\vee_{\scriptscriptstyle \mathbf{W}} z)=(\neg x\to y)\vee_{\scriptscriptstyle \mathbf{W}} (\neg x\to z)=(x\oplus y)\vee_{\scriptscriptstyle g(\mathbf{W})} (x\oplus z)$ by (QW*12).

Hence $g(\textbf{W})=\langle W;\oplus,{-},^{+},^{-},0,1\rangle$ is a quasi-MV* algebra.
\end{proof}

According to Proposition \ref{p3.3},  given a quasi-Wajsberg* algebra $\mathbf{W}=\langle W;\to,\neg,^{+},^{-},$ $1\rangle$, we define $x\oplus y=\neg x\to y$, then $\langle W;\oplus,\neg,^{+},^{-},0, 1\rangle$ is a quasi-MV* algebra where $0=x\to x$. Hence we have the following example.

\begin{exam}
Let $\mathbf{R}^*= \langle R^{*}; \to, \neg, ^{+}, ^{-}, 1_{\scriptscriptstyle \mathbf{R}^*} \rangle$ be a quasi-Wajsberg* algebra defined in Example 3.1.
Define $x\oplus y=\neg x\to y$ and $0=x\to x$ for any $x,y \in R^*$. Then $\langle R^{*}; \oplus, \neg, ^{+}, ^{-}, 0, 1_{\scriptscriptstyle \mathbf{R}^*} \rangle$ is a quasi-MV* algebra which is different from Example 2.2.
\end{exam}

\begin{thm}\label{th2} The following conclusions hold.
\begin{itemize}
  \item Let $\mathbf{A}=\langle A;\oplus,{-},^{+},^{-},0,1\rangle$ be a quasi-MV* algebra. Then $gf(\mathbf{A})=\mathbf{A}$.
  \item Let $\mathbf{W}=\langle W;\to,\neg,^{+},^{-},1\rangle$ be a quasi-Wajsberg* algebra. Then $fg(\mathbf{W})=\mathbf{W}$.
\end{itemize}
So $f$ and $g$ are mutually inverse correspondence.
\end{thm}

\begin{proof}
For any $x\in A$, we have that $-_{\scriptscriptstyle gf(\textbf{A})} x=\neg_{\scriptscriptstyle f(\textbf{A})} x =-_{\scriptscriptstyle \textbf{A}} x$ and $x\oplus_{\scriptscriptstyle gf(\textbf{A})} y=-_{\scriptscriptstyle gf(\textbf{A})}(-_{\scriptscriptstyle gf(\textbf{A})} x)\oplus_{\scriptscriptstyle gf(\textbf{A})} y=\neg_{\scriptscriptstyle f(\textbf{A})} x\to_{\scriptscriptstyle f(\textbf{A})} y=x\oplus_{\scriptscriptstyle\textbf{A}} y$ by Proposition \ref{p3.2},  so $gf(\textbf{A})=\textbf{A}$. Similarly, we have  $\neg_{\scriptscriptstyle f g(\textbf{W})} x=-_{\scriptscriptstyle g(\textbf{W})} x=\neg_{\scriptscriptstyle \textbf{W}} x$ and $x\to_{\scriptscriptstyle f g(\textbf{W})} y=\neg_{\scriptscriptstyle f g(\textbf{W})} \neg_{\scriptscriptstyle f g(\textbf{W})} x$ $\to_{\scriptscriptstyle f g(\textbf{W})} y=-_{\scriptscriptstyle g(\textbf{W})} x\oplus_{\scriptscriptstyle g(\textbf{W})} y=x\to_{\scriptscriptstyle \textbf{W}} y$ by Proposition \ref{p3.3}, so we have $fg(\textbf{W})=\textbf{W}$.
\end{proof}

Since quasi-Wajsberg* algebras are term equivalent to quasi-MV* algebras, we can obtain the term equivalence of MV*-algebras similarly.
The algebras which are term equivalent to MV*-algebras are named after Wajsberg*-algebras.

\begin{defn}\label{d1.3}
Let $\textbf{M}=\langle M;\to,\neg,1\rangle$ be an algebra of type $\langle2,1,0\rangle$. If the following conditions are satisfied for any ${x,y,z}\in M$,

(M1) $x\to y=\neg y\to \neg x$,

(M2) $(x\to 1)\to((y\to 1)\to z)=(y\to 1)\to((x\to 1)\to z)$,

(M3) $(1\to x)\to 1=1$,

(M4) $(y\to y)\to x=x$,

(M5) $x\to y=(y^{+}\to x^{-})\to (x^{+}\to y^{-})$,

(M6) $\neg(x\to y)= y\to x$,

(M7) $\neg\neg x=x$,

(M8) $(x\to (\neg x\to y))^{+}=x^{+}\to (\neg x^{+}\to y^{+})$,

(M9) $x\vee y=y\vee x$,

(M10) $x\vee (y\vee z)=(x\vee y)\vee z$,

(M11) $x\to (y\vee z)=(x\to y)\vee (x\to z)$,

\noindent in which $x^{+}=(x\to 1)\to 1$, $x^{-}=(x\to \neg 1)\to \neg 1$, $x\vee y = ((x^{+}\to y^{+})^{+}\to(\neg x)^{-}) \to((y^{-}\to x^{-})^{-}\to x^{-})$, then $\textbf{M}=\langle M;\to,\neg,1\rangle$ is called a \emph{Wajsberg*-algebra}.
\end{defn}

\begin{coro}\label{cor}
For any MV*-algebra $\mathbf{A}=\langle A;\oplus,{-},0,1\rangle$ and Wajsberg*-algebra $\mathbf{M}=\langle M;\to,\neg,1\rangle$. We have

\emph{(1)} If for any $x,y\in A$, we define $x\to y=-x\oplus y$ and $\neg x=-x$, then $f(\mathbf{A})=\langle A;\to,\neg,1\rangle$ is a Wajsberg*-algebra.

\emph{(2)} If for any $x,y\in M$, we define $0=x\to x$, $x\oplus y=\neg x\to y$ and $-x=\neg x$, then $g(\mathbf{M})=\langle M;\oplus,{-},0,1\rangle$ is an MV*-algebra.

\emph{(3)} $gf(\mathbf{A})=\mathbf{A}$, $fg(\mathbf{M})=\mathbf{M}$, and $f$ and $g$ are mutually inverse correspondence.
\end{coro}

Moreover, according to the equivalences in Theorem \ref{th2} and Corollary \ref{cor}, we can restate the
known algebraic theorems for quasi-MV* algebras (MV*-algebras) in terms of quasi-Wajsberg* algebras (Wajsberg*-algebras) without providing any additional proof.

\begin{coro}\label{l0.4}
Let $\mathbf{W}$ be a quasi-Wajsberg* algebra. Then for any $x,y,u,v\in W$, we have

\emph{(1)} If $x\le y$ and $y\le x$, then $0\to x=0\to y$,

\emph{(2)} $\neg 1\le x\le 1$,

\emph{(3)} $x\le 0\to x$ and $0\to x\le x$,

\emph{(4)} If $x\le y$ and $u\le v$, then $x\vee u\le y\vee v$,

\emph{(5)} If $x\le y$ and $u\le v$, then $y\wedge u\le x\wedge v$,

\emph{(6)} $x\le y$ iff $x\wedge y=0\to x$,

\emph{(7)} If $x\le y$, then $\neg y\le \neg x$,

\emph{(8)} If $x\le y$ and $u\le v$, then $y \to u\le x\to v$,

\emph{(9)} If $x\le y$, then $x^+\le y^+$ and $x^-\le y^-$,

\emph{(10)} $x^-\le x\le x^+$.
\end{coro}

\section{Direct products and congruences of quasi-Wajsberg* algebras}
In this section, we study the properties about direct products of quasi-Wajsberg* algebras, and discuss several special congruences on quasi-Wajsberg* algebras and the properties of the quotient algebras.

Let $\textbf{W} = \langle W; \to_{\scriptscriptstyle W}, \neg_{\scriptscriptstyle W}, ^{+_{\scriptscriptstyle W}}, ^{-_{\scriptscriptstyle W}}, 1_{\scriptscriptstyle W} \rangle$ and $\textbf{V}= \langle V; \to_{\scriptscriptstyle V}, \neg_{\scriptscriptstyle V}, ^{+_{\scriptscriptstyle V}}, ^{-_{\scriptscriptstyle V}}, 1_{\scriptscriptstyle V} \rangle$ be quasi-Wajsberg* algebras. If for any $\langle x,y\rangle, \langle u,v\rangle\in W\times V$, we define $\neg\langle x,y\rangle=\langle \neg_{\scriptscriptstyle W} x,\neg_{\scriptscriptstyle V} y\rangle$, $\langle x,y\rangle\to\langle u,v\rangle =\langle x\to_{\scriptscriptstyle W} u, y\to_{\scriptscriptstyle V} v\rangle$, ${\langle x,y\rangle}^{+}=\langle x^{+_{W}},y^{+_{V}}\rangle$ and ${\langle x,y\rangle}^{-}=\langle x^{-_{W}},y^{-_{V}}\rangle$, then $\textbf{W}\times \textbf{V}=\langle W\times V; \to, \neg, ^{+}, ^{-}, \langle 1_{\scriptscriptstyle W},1_{\scriptscriptstyle V}\rangle\rangle$ is the \emph{direct product} of $\textbf{W}$ and $\textbf{V}$.
Moreover, $\textbf{W}\times \textbf{V}=\langle W\times V; \to, \neg, ^{+}, ^{-}, \langle 1_{\scriptscriptstyle W},1_{\scriptscriptstyle V}\rangle\rangle$ is a quasi-Wajsberg* algebra.

\begin{coro} Let $\{\mathbf{W}_i: i\in I\}$ be a family of quasi-Wajsberg* algebras. Then $\prod_{i\in I}\mathbf{W}_i$ is a quasi-Wajsberg* algebra.
\end{coro}

Let $\mathbf{W}$ be a quasi-Wajsberg* algebra. We denote $R(W)=\{x\in W: 0\to x=x\}$, from Proposition \ref{pr.1}(1) and Lemma \ref{le.1}(1), we can get that $1,\neg 1, x\to y, \neg(x\to y)\in R(W)$.
Moreover, we have the following results.

\begin{prop}\label{p0.1}
Let $\mathbf{W}$ be a quasi-Wajsberg* algebra. For any $x,y\in W$, we define
\[\langle x,y\rangle \in \mu\ \text{iff}\ x\le y \ \text{and}\ y\le x,\]
\[\langle x,y\rangle \in\tau\ \text{iff}\ x=y \ \text{or}\ x,y\in R(W).\]
Then $\mu$ and $\tau$ are congruences on $\mathbf{W}$.
\end{prop}

\begin{proof}
We can easily get that $\mu, \tau$ are equivalence relations. For any $\langle x,y\rangle\in \mu$, we have $x\le y$ and $y\le x$, it turns out that $\neg x\le \neg y$ and $\neg y\le \neg x$ by Corollary \ref{l0.4}(7), so $\langle \neg x,\neg y\rangle \in \mu$. For any $\langle x,y\rangle, \langle u,v\rangle\in \mu$, we have $x\le y$ and $y\le x$, $u\le v$ and $v\le u$, so $x\to u\le y\to v$ and $y\to v\le x\to u$ by Corollary \ref{l0.4}(8), thus $\langle x\to u, y\to v\rangle \in\mu$. For any $\langle x,y\rangle\in \mu$, then $x\le y$ and $y\le x$, we have $x^+\le y^+$, $x^-\le y^-$, $y^+\le x^+$ and $y^-\le x^-$, so $\langle x^+,y^+\rangle$ and $\langle x^-,y^-\rangle \in \mu$. Hence, the relation $\mu$ is a congruence on $\textbf{W}$. Similarly, for any $\langle x,y\rangle\in \tau$, we have $x=y$ or $x,y\in R(W)$, and then we have $\neg x=\neg y$ or $\neg x,\neg y\in R(W)$, so $\langle \neg x,\neg y\rangle \in \tau$. For any $\langle x,y\rangle, \langle u,v\rangle\in \tau$, we can get $x\to u=y\to v$ or $x\to u,y\to v\in R(W)$.
For any $\langle x,y\rangle\in \tau$, then $x=y$ or $x, y\in R(W)$. If $x=y$, then $x^+=y^+$, we have $\langle x^+,y^+\rangle \in \tau$. If $x,y\in R(W)$, then $0\to x^+=(0\to x)^+=x^+$ and $0\to y^+=(0\to y)^+=y^+$, so $\langle x^+, y^+\rangle \in\tau$. Similarly, we have $\langle x^-, y^-\rangle \in\tau$. Hence $\tau$ is also a congruence on $\textbf{W}$.
\end{proof}

\begin{prop}\label{l0.5}
Let $\mathbf{W}$ be a quasi-Wajsberg* algebra. For any $x,y\in W$, we have

\emph{(1)} $\langle x, y \rangle\in \mu$ iff $0\to x=0\to y$,

\emph{(2)} $\mu \cap \tau=\Delta$ $($the diagonal relation on $\mathbf{W}$\emph{)},

\emph{(3)} denote $G(\mu\cup\tau)=\bigcup\{\theta_1\circ\theta_2\circ\cdots\theta_k$, where $\theta_i=\mu$ or $\theta_i=\tau$, $i=1,2,\cdots,k$ and $k<\infty \}$, then $G(\mu\cup\tau)=\nabla$ $($the all relation on $\mathbf{W}$\emph{)},

\emph{(4)} $\mu \circ \tau=\nabla$ iff $x/\mu \cap y/\tau\neq \varnothing$.
\end{prop}

\begin{proof}
(1) We have $\langle x,y\rangle\in \mu$ iff $x\le y$ and $y\le x$ iff $0\to x=x\wedge y=y\wedge x=0\to y$ by Corollary \ref{l0.4}(6).

(2) We have $\langle x,y\rangle\in \mu\cap\tau$ iff $\langle x,y\rangle\in \mu$ and $\langle x,y\rangle\in \tau$ iff $x=y$ iff $\langle x,y\rangle\in \Delta$.

(3) It is clear that $G(\mu\cup\tau)\subseteq\nabla$. We prove that $\nabla\subseteq G(\mu\cup\tau)$ by the following cases. For any $\langle x,y\rangle\in\nabla$,
if $x,y\in R(W)$, we have $\langle x,y\rangle\in \tau \subseteq G(\mu\cup\tau)$. If $x,y\notin R(W)$, then we have $\langle x,0\to x\rangle \in \mu$, $\langle 0\to x, 0\to y\rangle \in \tau$ and $\langle 0\to y, y\rangle \in \mu$ by Corollary \ref{l0.4}(3) and Proposition \ref{p0.1}, so $\langle x,y\rangle\in \mu\circ\tau\circ\mu \subseteq G(\mu\cup\tau)$. If $x\in R(W)$ and $y\notin R(W)$, we have $\langle x, 0\to y\rangle \in \tau$ and $\langle 0\to y, y\rangle \in \mu$, so $\langle x,y\rangle\in \tau\circ\mu \subseteq G(\mu\cup\tau)$. If $x\notin R(W)$ and $y\in R(W)$, we have $\langle x, 0\to x\rangle \in \mu$ and $\langle 0\to x, y\rangle \in \tau$, so $\langle x,y\rangle\in \mu\circ\tau \subseteq G(\mu\cup\tau)$. Therefore, $G(\mu\cup\tau)=\nabla$.

(4) We have that $x/\mu \cap y/\tau\neq \varnothing$ iff there exists $z\in W$ such that $\langle x,z\rangle \in \mu$ and $\langle z,y\rangle \in \tau$ iff $\langle x,y\rangle\in \mu\circ \tau$.
\end{proof}

Let $\textbf{W}$ be a quasi-Wajsberg* algebra and  $\theta$ be a congruence on $\textbf{W}$. For any $x\in W$, the equivalence class of $x$ with respect to $\theta$ is denoted by $x/\theta=\{y\in W|\langle x,y\rangle\in \theta\}$ and the set of all equivalence classes is denoted by $W/\theta$. For any $x/\theta,y/\theta\in W/\theta$, we define $\neg_{\scriptscriptstyle W/\theta}(x/\theta)=(\neg x)/\theta$, $(x/\theta)\to_{\scriptscriptstyle W/\theta}(y/\theta)=(x\to y)/\theta$, $(x/\theta)^{+_{W/\theta}}=x^+/\theta$ and $(x/\theta)^{-_{W/\theta}}=x^-/\theta$. Then $\textbf{W}/\theta=\langle W/\theta; \to_{\scriptscriptstyle W/\theta}, \neg_{\scriptscriptstyle W/\theta}, ^{+_{W/\theta}}, ^{-_{W/\theta}}, 1/\theta \rangle$ is a quasi-Wajsberg* algebra.

\begin{prop}\label{p0.2}
Let $\mathbf{W}$ be a quasi-Wajsberg* algebra. Then $\mathbf{W}/\mu=\langle W/\mu; \to, \neg, 1/\mu \rangle$ is a Wajsberg*-algebra.
\end{prop}

\begin{proof}
We only need to verify that $0\to x/\mu=(0\to x)/\mu=x/\mu$ for any $x/\mu \in W/\mu$. Indeed, $z\in (0\to x)/\mu$ iff $z\le 0\to x$ and $0\to x\le z$ iff $z\le x$ and $x\le z$ iff $z\in x/\mu$ by Proposition \ref{l0.5}(1) Corollary \ref{l0.4}(3).
\end{proof}

\begin{lem}\label{l0.6}
Let $\mathbf{W}$ be a quasi-Wajsberg* algebra. Then we have

\emph{(1)} $\neg(1/\tau)=1/\tau$,

\emph{(2)} $(x/\tau)\vee (y/\tau)=0/\tau=(x/\tau)\wedge (y/\tau)=(x/\tau)\rightarrow (y/\tau)$.

\begin{proof}

(1) Since $1, \neg 1\in R(W)$ by Proposition \ref{pr.1}(2), we have $\langle \neg1,1 \rangle \in \tau$ and then $\neg(1/\tau)=1/\tau$.

(2) Since $(x/\tau)\vee (y/\tau)=(x\vee y)/\tau$, $(x/\tau)\wedge (y/\tau)=(x\wedge y)/\tau$ and $(x/\tau)\rightarrow (y/\tau) =(x\to y)/\tau$, from Lemma \ref{le.1}(1) and Proposition \ref{pr.3}(4), we have $x\to y, x\vee y, x\wedge y \in R(W)$. Hence $\langle 0, x\to y \rangle, \langle 0, x\vee y \rangle, \langle 0, x\wedge y \rangle \in \tau$, it follows that $(x/\tau)\vee (y/\tau)=0/\tau=(x/\tau)\wedge (y/\tau)=(x/\tau)\rightarrow (y/\tau)$.
\end{proof}
\end{lem}

\begin{defn}\label{}
Let $\textbf{W}$ be a quasi-Wajsberg* algebra. Then $\textbf{W}$ is called \emph{flat} if the equation $0=1$ is satisfied.
\end{defn}

\begin{prop}\label{p0.3}
Let $\mathbf{W}$ be a quasi-Wajsberg* algebra. Then $\mathbf{W}/\tau=\langle W/\tau; \to, \neg, 1/\tau \rangle$ is a flat quasi-Wajsberg* algebra.
\end{prop}
\begin{proof}
Since $0, 1\in R(W)$, we have $\langle0,1\rangle \in \tau$ and then $0/\tau=1/\tau$, so $\textbf{W}/\tau$ is flat.
\end{proof}

\begin{prop}\label{CL7}
Let $\mathbf{W}$ be a quasi-Wajsberg* algebra. Then there exist a Wajsberg*-algebra $\mathbf{M}$ and a flat quasi-Wajsberg* algebra $\mathbf{F}$ such that $\mathbf{W}$ can be embedded into the direct product $\mathbf{M}\times\mathbf{F}$.
\end{prop}

\begin{proof}
Define $\textbf{M}=\textbf{W}/\mu$ and $\textbf{F}=\textbf{W}/\tau$. Then we have that $\textbf{M}$ is a Wajsberg*-algebra by Proposition \ref{p0.2} and $\textbf{F}$ is a flat quasi-Wajsberg* algebra by Proposition \ref{p0.3}. Now we prove that the mapping $h:\textbf{W}\to\textbf{M}\times\textbf{F}$ is injective. For any $x,y\in W$, if $h(x)=h(y)$, we have $\langle x/\mu,x/\tau\rangle =\langle y/\mu,y/\tau\rangle $, it follows that $x/\mu=y/\mu$ and $x/\tau=y/\tau$, so $\langle x,y\rangle\in\mu$ and $\langle x,y\rangle\in\tau$. From Proposition \ref{l0.5}(2), we have $x=y$. Thus $f$ is an embedding mapping.
\end{proof}

\begin{defn}\label{}
Let $\textbf{W}$ be a quasi-Wajsberg* algebra. Then $\textbf{W}$ is called \emph{linear} if for any $x,y\in W$, either $x\le y$ or $y\le x$.
\end{defn}

\begin{coro}\label{CL8}
Let $\mathbf{W}$ be a quasi-Wajsberg* algebra. Then $\mathbf{W}/\tau=\langle W/\tau; \to_{W/\tau}, \neg_{W/\tau}, 1/\tau \rangle$ is a linear quasi-Wajsberg* algebra.

\begin{proof}
From Lemma \ref{l0.6}(1) and Proposition \ref{p0.3}, we have $(\neg1)/\tau = 0/\tau=1/\tau$. For any $x/\tau, y/\tau\in W/\tau$, we have $0/\tau \leq x/\tau \leq 0/\tau$ and $0/\tau \leq y/\tau \leq 0/\tau$ from Corollary \ref{l0.4}(2). It turns out that $x/\tau \leq 0/\tau \leq y/\tau$ and $y/\tau \leq 0/\tau \leq x/\tau$, i.e., $x/\tau \leq y/\tau$ and $y/\tau \leq x/\tau$. Hence $\textbf{W}/\tau$ is linear.
\end{proof}
\end{coro}

In \cite{c2}, author proved that any MV*-algebra can be embedded into the direct product of linear MV*-algebras. Based on the term equivalence of MV*-algebras, we have the following result.

\begin{lem}\label{l1111}
Any Wajsberg*-algebra is representable as a subdirect product of linear Wajsberg*-algebras.
\end{lem}

\begin{prop}\label{CL9}
Every quasi-Wajsberg* algebra is representable as a subdirect product of linear quasi-Wajsberg* algebras.
\end{prop}

\begin{proof} Let $\textbf{W}$ be a quasi-Wajsberg* algebra. Then there exist a Wajsberg*-algebra $\textbf{M}$ and a flat quasi-Wajsberg* algebra $\textbf{F}$
such that $\textbf{W}$ can be embedded into $\textbf{M}\times\textbf{F}$ by Proposition \ref{CL7}. According to Lemma \ref{l1111}, we have that $\textbf{M}$ can be embedded into a direct product of $\{\textbf{M}_i\}_{i\in I}$ which are linear Wajsberg*-algebras.  Since there is an isomorphism
$$(\prod_{i\in I}\textbf{M}_i)\times\textbf{F}\to \prod_{i\in I}(\textbf{M}_i\times\textbf{F}),$$ where
$$(x_1, x_2, \cdots, x_n,y )\mapsto((x_1, y),(x_2, y), \cdots, (x_n, y)),$$
we have that $\textbf{W}$ is embedded into $\prod_{i\in I}(\textbf{M}_i\times\textbf{F})$. In addition, since any Wajsberg*-algebra is a quasi-Wajsberg* algebra, we get that $\textbf{M}_i$ is linear quasi-Wajsberg* algebra. Meanwhile, $\textbf{F}$ is also a linear quasi-Wajsberg* algebra by Corollary \ref{CL8}. It follows that $\textbf{M}_i\times\textbf{F}$ is a linear quasi-Wajsberg* algebra. Hence $\textbf{W}$ can be embedded into the direct product of linear quasi-Wajsberg* algebras.
\end{proof}

\section{The Logic associated with quasi-MV* algebras}\label{Sec-L}
In this section, we study the associated logical system of quasi-MV* algebras and call it $\text{q}\L^{*}$.

Denote $V=\{p_1, p_2, \ldots \}$ a set of all propositional variables and $F(V)$ the formulas set generated by $V$ with logical connectives $\to$, $\neg$, $^{+}$, $^{-}$ and a constant $1$. Then $\langle F(V); \to, \neg, ^{+}, ^{-}, 1 \rangle$ is a free algebra.
For any $p, q\in F(V)$, the notation $p\leftrightarrow q$ stands for $p\to q$ and $q\to p$.

The axioms of $\text{q}\L^{*}$ are defined as follows.

\noindent\textbf{Axioms schemas}


(Q1) $(p\to q)\leftrightarrow (\neg q\to \neg p)$,

(Q2) $1\leftrightarrow ((1\to p)\to 1)$,

(Q3) $p\leftrightarrow ((q\to q)\to p)$,

(Q4) $(p\to q)\leftrightarrow((q^{+}\to p^{-})\to (p^{+}\to q^{-}))$,

(Q5) $\neg(p\to q)\leftrightarrow (q\to p)$,

(Q6) $(p\to (\neg p\to q))^+\leftrightarrow (p^+\to (\neg p^+\to q^+))$,

(Q7) $(p\to (q\vee r))\leftrightarrow((p\to r)\vee (p\to q))$ where $p\vee q = ((p^{+}\to q^{+})^{+})\to(\neg p)^{-} )\to((q^{-}\to p^{-})^{-}\to p^{-})$,

(Q8) $(p\vee (q\vee r))\leftrightarrow ((p\vee q)\vee r)$,

(Q9) $((p\to 1)\to((q\to 1)\to r))\to ((q\to 1)\to((p\to 1)\to r))$,

(Q10) $p\to 1$,

(Q11) $(1\to 1)\to p^{+}\leftrightarrow(p\to 1)\to 1$ and $(1\to 1)\to p^{-}\leftrightarrow(p\to \neg 1)\to \neg 1$.

The deduction rules of $\text{q}\L^{*}$ are as follows.

\noindent\textbf{Rules of deduction}

(R1) $p, p\to q\vdash_{\tiny\text{q}\L^{*}} ((r\to r)\to q)$

(R2) $((r\to r)\to(p\to q))\vdash_{\tiny\text{q}\L^{*}} (p\to q)$,

(R3) $p\to q, r\to t\vdash_{\tiny\text{q}\L^{*}} ((q\to r)\to (p\to t))$.

\begin{defn}
Let $F(V)$ be the formulas set of $\text{q}\L^{*}$. If $q_1,q_2,\ldots,q_n$ ($n\ge 1$) is a sequence of formulas such that one of the following cases holds:

(1) $q_i(i\le n)$ is an axiom,

(2) $q_i(i\le n)$ is obtained by $q_j$ and $q_k$ for some $j,k<i$ with the \emph{Rules of deduction},

\noindent then the sequence $q_1,q_2,\ldots,q_n$ is called a \emph{proof} of $q_n$ and denoted by $\vdash_{\tiny\text{q}\L^{*}} q_n$, in which $q_n$ is called a \emph{theorem}. The set of theorems in q$\L^{*}$ is denoted by \emph{Theor}$_{\tiny\text{q}\L^{*}}$.
\end{defn}

\begin{defn}
Let $\Gamma\subseteq F(V)$ be a subset of formulas of $\text{q}\L^{*}$. If $q_1,q_2,\ldots,q_n$ ($n\ge 1$) is a sequence of formulas such that one of the following cases holds:

(1) $q_i(i\le n)$ is an axiom,

(2) $q_i\in \Gamma(i\le n)$,

(3) $q_i(i\le n)$ is obtained by $q_j$ and $q_k$ for some $j,k<i$ with the \emph{Rules of deduction},

\noindent then the sequence $q_1,q_2,\ldots,q_n$ is called a \emph{proof from $\Gamma$ to $q_n$} and denoted by $\Gamma\vdash_{\tiny\text{q}\L^{*}} q_n$. Moreover,  $q_n$ is called the probable formula from $\Gamma$.
\end{defn}

For convenience, we abbreviate the symbol $\vdash_{\tiny\text{q}\L^{*}}$ as $\vdash$ in the following. Besides, the set of all probable formulas from $\Gamma$ is denoted by ${\Gamma}^{\vdash}$.

\begin{defn}
Let $\Theta \subseteq F(V)$ be a subset of formulas of $\text{q}\L^{*}$. If $\Theta$ is closed under the rules of deduction (i.e., ${\Theta}^{\vdash}=\Theta$), then $\Theta$ is called a q$\L^{*}$-\emph{theory}.
\end{defn}

\begin{prop}\label{QL1}
Let $F(V)$ be the formulas set of q$\L^{*}$. Then the following hold for any $p, q, r, t\in F(V)$,

\emph{(1)} If $\vdash p\leftrightarrow q$, then $\vdash \neg p\leftrightarrow \neg q$,

\emph{(2)} If $\vdash p\leftrightarrow q$ and $\vdash r\leftrightarrow t$, then $\vdash (p\to r)\leftrightarrow(q\to t)$,

\emph{(3)} If $\vdash p\leftrightarrow q$ and $\vdash q\leftrightarrow r$, then $\vdash p\leftrightarrow r$,

\emph{(4)} $\vdash \neg (p\to q)\leftrightarrow (\neg p\to \neg q)$,

\emph{(5)} $\vdash p\leftrightarrow p$,

\emph{(6)} If $\vdash p_1\leftrightarrow r_1$, then $\vdash p\leftrightarrow r$, where $p_1$ is a subformula of $p$ and $r$ is obtained by replacing $p_1$ in $p$ with $r_1$,

\emph{(7)} $\vdash (p\to p)\leftrightarrow (q\to q)$,

\emph{(8)} $\vdash \neg\neg(p\to p)\leftrightarrow (p\to p)$,

\emph{(9)} $\vdash p\leftrightarrow \neg\neg p$,

\emph{(10)} $\vdash (\neg p\to q)\leftrightarrow (\neg q\to p)$,

\emph{(11)} $\vdash (\neg p)^{+}\leftrightarrow \neg p^{-}$ and $\vdash (\neg p)^{-}\leftrightarrow \neg p^{+}$,

\emph{(12)} If $\vdash p\leftrightarrow q$, then $\vdash p^+\leftrightarrow q^+$ and $\vdash p^-\leftrightarrow q^-$.
\end{prop}

\begin{proof}
\begin{enumerate}
\renewcommand{\labelenumi}{(\theenumi)}
\item If $\vdash p\leftrightarrow q$, then
\begin{flalign*}
  1^{\circ} &\ p \to q&(\text{Hypothesis})\\
  2^{\circ} &\ (p\to q)\to (\neg q\to \neg p)& (\text{Q}1)\\
  3^{\circ} &\ (1\to 1)\to (\neg q\to \neg p) & 1^{\circ}, 2^{\circ}, (\text{R}1)\\
  4^{\circ} &\ \neg q\to \neg p & 3^{\circ},(\text{R}2)
\end{flalign*}
Similarly,
\begin{flalign*}
  1^{\circ} & \ q \to p&(\text{Hypothesis})\\
  2^{\circ} & \ (q\to p)\to (\neg p\to \neg q)& (\text{Q}1)\\
  3^{\circ} & \ (1\to 1)\to (\neg p\to \neg q) & 1^{\circ}, 2^{\circ}, (\text{R}1)\\
  4^{\circ} & \ \neg p\to \neg q & 3^{\circ},(\text{R}2)
\end{flalign*}
So $\vdash \neg p\leftrightarrow \neg q$.

\item If $\vdash p\leftrightarrow q$ and $\vdash r\leftrightarrow t$, then
\begin{flalign*}
  1^{\circ} & \ p\to q &(\text{Hypothesis})\\
  2^{\circ} & \ t\to r &(\text{Hypothesis})\\
  3^{\circ} & \ (q\to t)\to (p\to r)& 1^{\circ}, 2^{\circ}, (\text{R}3)
\end{flalign*}
Similarly,
\begin{flalign*}
  1^{\circ} & \ q\to p& (\text{Hypothesis})\\
  2^{\circ} & \ r\to t& (\text{Hypothesis})\\
  3^{\circ} & \ (p\to r)\to(q\to t) & 1^{\circ}, 2^{\circ}, (\text{R}3)
\end{flalign*}
So $\vdash (p\to r)\leftrightarrow(q\to t)$.

\item If $\vdash p\leftrightarrow q$ and $\vdash q\leftrightarrow r$, then
\begin{flalign*}
  1^{\circ} & \ p\to q &(\text{Hypothesis})\\
  2^{\circ} & \ q\to r &(\text{Hypothesis})\\
  3^{\circ} & \ (q\to q)\to (p\to r)& 1^{\circ}, 2^{\circ}, (\text{R}3)\\
  4^{\circ} & \ ((q\to q)\to (p\to r))\to (p\to r) & (\text{Q}3)\\
  5^{\circ} & \ (1\to 1)\to (p\to r) & 3^{\circ}, 4^{\circ}, (\text{R}1)\\
  6^{\circ} & \ p\to r &5^{\circ}, (\text{R}2)
\end{flalign*}
As same as the former,
\begin{flalign*}
  1^{\circ} & \ r\to q &(\text{Hypothesis})\\
  2^{\circ} & \ q\to p&(\text{Hypothesis})\\
  3^{\circ} & \ (q\to q)\to (r\to p)& 1^{\circ}, 2^{\circ}, (\text{R}3)\\
  4^{\circ} & \ ((q\to q)\to (r\to p))\to (r\to p) & (\text{Q}3)\\
  5^{\circ} & \ (1\to 1)\to (r\to p) & 3^{\circ}, 4^{\circ}, (\text{R}1)\\
  6^{\circ} & \ r\to p &5^{\circ}, (\text{R}2)
\end{flalign*}
So $\vdash p\leftrightarrow r$.

\item We have
\begin{flalign*}
  1^{\circ} & \ \neg (p\to q)\to ( q\to p) & (\text{Q}5)\\
  2^{\circ} & \ (q\to p)\to (\neg p\to \neg q) & (\text{Q}1)\\
  3^{\circ} & \ \neg (p\to q)\to (\neg p\to \neg q) & 1^{\circ}, 2^{\circ}, (3)
\end{flalign*}
Similarly,
\begin{flalign*}
  1^{\circ} & \ (\neg p\to \neg q)\to ( q\to p) & (\text{Q}1)\\
  2^{\circ} & \ (q\to p)\to \neg(p\to q) & (\text{Q}5)\\
  3^{\circ} & \ (\neg p\to \neg q)\to \neg (p\to q) & 1^{\circ}, 2^{\circ}, (3)
\end{flalign*}
So $\vdash \neg (p\to q)\leftrightarrow (\neg p\to \neg q)$.

\item We have
\begin{flalign*}
  1^{\circ} & \ p\to ((q\to q)\to p) &(\text{Q}3)\\
  2^{\circ} & \ ((q\to q)\to p)\to p & (\text{Q}3)\\
  3^{\circ} & \ p\to p & 1^{\circ}, 2^{\circ}, (3)
\end{flalign*}
So $\vdash p\leftrightarrow p$.

\item Suppose that $\vdash p_1 \leftrightarrow r_1$. We use the induction for the number $n$ of connectives in $p$. If $n=0$, then $p=p_1$ and then $r=r_1$, so we have $\vdash p \leftrightarrow r$. If $n=1$, then $p_1$ is a propositional variable, we have

(a) If $p=\neg p_1$, then $r=\neg r_1$.
\begin{flalign*}
  1^{\circ} & \ p_1\leftrightarrow r_1 &(\text{Hypothesis})\\
  2^{\circ} & \ \neg p_1\leftrightarrow \neg r_1 & 1^{\circ}, (1)
\end{flalign*}
So $\vdash p \leftrightarrow r$.

(b) If $p={p_1}^{+}$, then $r={r_1}^{+}$.
\begin{flalign*}
  1^{\circ} & \ p_1\leftrightarrow r_1 & (\text{Hypothesis})\\
  2^{\circ} & \ 1\leftrightarrow 1 & (5)\\
  3^{\circ} & \ (p_1\to 1)\leftrightarrow (r_1\to 1) & 1^{\circ}, 2^{\circ}, (2)\\
  4^{\circ} & \ ((p_1\to 1)\to 1)\leftrightarrow ((r_1\to 1)\to 1) & 3^{\circ}, 2^{\circ}, (2)\\
  5^{\circ} & \ ((1\to 1)\to {p_1}^{+})\leftrightarrow ((p_1\to 1)\to 1) & (\text{Q}11)\\
  6^{\circ} & \ ((1\to 1)\to {p_1}^{+})\leftrightarrow ((r_1\to 1)\to 1) & 5^{\circ}, 4^{\circ}, (3)\\
  7^{\circ} & \ ((r_1\to 1)\to 1)\leftrightarrow ((1\to 1)\to {r_1}^{+}) & (\text{Q}11)\\
  8^{\circ} & \ ((1\to 1)\to {p_1}^{+})\leftrightarrow ((1\to 1)\to {r_1}^{+}) & 6^{\circ}, 7^{\circ}, (3)\\
  9^{\circ} & \ {p_1}^{+}\leftrightarrow ((1\to 1)\to {p_1}^{+}) & (\text{Q}3)\\
  10^{\circ} & \ {p_1}^{+}\leftrightarrow ((1\to 1)\to {r_1}^{+}) & 9^{\circ}, 8^{\circ}, (3)\\
  11^{\circ} & \ ((1\to 1)\to {r_1}^{+})\leftrightarrow {r_1}^{+}& (\text{Q}3)\\
  12^{\circ} & \ {p_1}^{+}\leftrightarrow {r_1}^{+}& 10^{\circ}, 11^{\circ}, (3)
\end{flalign*}
So we have $\vdash p\leftrightarrow r$.

(c) If $p={p_1}^{-}$, then $r={r_1}^{-}$.
\begin{flalign*}
  1^{\circ} & \ p_1\leftrightarrow r_1 & (\text{Hypothesis})\\
  2^{\circ} & \ \neg 1\leftrightarrow \neg 1 & (5)\\
  3^{\circ} & \ (p_1\to \neg 1)\leftrightarrow (r_1\to \neg 1) & 1^{\circ}, 2^{\circ}, (2)\\
  4^{\circ} & \ ((p_1\to \neg1)\to \neg1)\leftrightarrow ((r_1\to \neg 1)\to \neg 1) & 3^{\circ}, 2^{\circ}, (2)\\
  5^{\circ} & \ ((1\to 1)\to {p_1}^{-})\leftrightarrow ((p_1\to \neg 1)\to \neg 1) & (\text{Q}11)\\
  6^{\circ} & \ ((1\to 1)\to {p_1}^{-})\leftrightarrow ((r_1\to \neg 1)\to \neg 1) & 5^{\circ}, 4^{\circ}, (3)\\
  7^{\circ} & \ ((r_1\to \neg 1)\to \neg 1)\leftrightarrow ((1\to 1)\to {r_1}^{-}) & (\text{Q}11)\\
  8^{\circ} & \ ((1\to 1)\to {p_1}^{-})\leftrightarrow ((1\to 1)\to {r_1}^{-}) & 6^{\circ}, 7^{\circ}, (3)\\
  9^{\circ} & \ {p_1}^{-}\leftrightarrow ((1\to 1)\to {p_1}^{-}) & (\text{Q}3)\\
  10^{\circ} & \ {p_1}^{-}\leftrightarrow ((1\to 1)\to {r_1}^{-}) & 9^{\circ}, 8^{\circ}, (3)\\
  11^{\circ} & \ ((1\to 1)\to {r_1}^{-})\leftrightarrow {r_1}^{-}& (\text{Q}3)\\
  12^{\circ} & \ {p_1}^{-}\leftrightarrow {r_1}^{-}& 10^{\circ}, 11^{\circ}, (3)
\end{flalign*}
So we have $\vdash p\leftrightarrow r$.

(d) If $p=p_1\to t_1$, then $r=r_1\to t_1$, where $t_1$ is a propositional variable.
\begin{flalign*}
  1^{\circ} & \ p_1\leftrightarrow r_1 &(\text{Hypothesis})\\
  2^{\circ} & \ t_1\leftrightarrow t_1 & (5)\\
  3^{\circ} & \ (p_1\to t_1)\leftrightarrow (r_1\to t_1) & 1^{\circ}, 2^{\circ}, (2)
\end{flalign*}
So we have $\vdash p\leftrightarrow r$.

Assume that the result holds when $n\le k$. Now we verify that the following cases hold when $n=k+1$.

(e) If $p=\neg u$, then $p_1$ is a subformula of $u$ and we denote that $v$ is obtained by replacing $p_1$ with $r_1$ in $u$, so $r=\neg v$. Since the number of connectives in $u$ is not more than $k$ and $\vdash p_1\leftrightarrow r_1$, we have $\vdash u\leftrightarrow v$ by the inductive assumption. So we have $\vdash p\leftrightarrow r$.

(f) If $p={u}^{+}$, then $p_1$ is a subformula of $u$ and we denote that $v$ is obtained by replacing $p_1$ with $r_1$ in $u$, so $r={v}^{+}$. Since the number of connectives in $u$ is not more than $k$ and $\vdash p_1\leftrightarrow r_1$, we have $\vdash u\leftrightarrow v$ by the inductive assumption. So $\vdash p\leftrightarrow r$.

(g) If $p={u}^{-}$, then $p_1$ is a subformula of $u$ and we denote that $v$ is obtained by replacing $p_1$ with $r_1$ in $u$, so $r={v}^{-}$. Since the number of connectives in $u$ is not more than $k$ and $\vdash p_1\leftrightarrow r_1$, we have $\vdash u\leftrightarrow v$ by the inductive assumption. So $\vdash p\leftrightarrow r$.

(h) If $p=u\to q$ such that $p_1$ is a subformula of $u$ or $q$, then we denote that $v$ or $t$ is obtained by replacing $p_1$ with $r_1$ in $u$ or $q$, so $r=v\to q$ or $r=u\to t$. Since the number of connectives in $u$ or $q$ is not more than $k$ and $\vdash p_1\leftrightarrow r_1$, we have $\vdash u\leftrightarrow v$ or $\vdash q\leftrightarrow t$ by the inductive assumption. So
\begin{flalign*}
  1^{\circ} & \ u\leftrightarrow v &(\text{Hypothesis})\\
  2^{\circ} & \ q\leftrightarrow q & (5)\\
  3^{\circ} & \ (u\to q)\leftrightarrow (v\to q) & 1^{\circ}, 2^{\circ}, (2)
\end{flalign*}
or
\begin{flalign*}
  1^{\circ} & \ u\leftrightarrow u & (5)\\
  2^{\circ} & \ q\leftrightarrow t&(\text{Hypothesis})\\
  3^{\circ} & \ (u\to q)\leftrightarrow (u\to t) & 1^{\circ}, 2^{\circ}, (2)
\end{flalign*}
Hence $\vdash p\leftrightarrow r$.

\item We have
\begin{flalign*}
  1^{\circ} & \ p\to p &(5)\\
  2^{\circ} & \ (p\to p)\to ((q\to q)\to (p\to p)) & (\text{Q}3)\\
  3^{\circ} & \ (1\to 1)\to ((q\to q)\to (p\to p)) & 1^{\circ}, 2^{\circ}, (\text{R}1)\\
  4^{\circ} & \ (q\to q)\to (p\to p) & 3^{\circ}, (\text{R}2)
\end{flalign*}
Swapping $p$ and $q$, we get $\vdash (p\to p)\leftrightarrow (q\to q)$.

\item We have
\begin{flalign*}
  1^{\circ} & \ \neg(p\to p)\leftrightarrow (p\to p)& (8)\\
  2^{\circ} & \ \neg\neg(p\to p)\leftrightarrow \neg(p\to p) & 1^{\circ}, (1)\\
  3^{\circ} & \ \neg\neg(p\to p)\leftrightarrow (p\to p) & 2^{\circ}, 1^{\circ}, (3)
\end{flalign*}
So $\vdash \neg\neg(p\to p)\leftrightarrow (p\to p)$.
\item We have
\begin{flalign*}
  1^{\circ} & \ p\leftrightarrow ((q\to q)\to p)& (\text{Q}3)\\
  2^{\circ} & \ ((q\to q)\to p)\leftrightarrow (\neg p\to \neg(q\to q)) & (\text{Q}1)\\
  3^{\circ} & \ p\leftrightarrow (\neg p\to \neg(q\to q)) & 1^{\circ}, 2^{\circ}, (3)\\
  4^{\circ} & \ (\neg p\to \neg(q\to q))\leftrightarrow (\neg\neg(q\to q)\to \neg\neg p) & (\text{Q}3)\\
  5^{\circ} & \ p\leftrightarrow (\neg\neg(q\to q)\to \neg\neg p) & 3^{\circ}, 4^{\circ}, (3)\\
  6^{\circ} & \ \neg\neg(q\to q)\leftrightarrow (q\to q) & (9)\\
  7^{\circ} & \ p\leftrightarrow ((q\to q)\to \neg\neg p) & 5^{\circ}, 6^{\circ}, (6)\\
  8^{\circ} & \ ((q\to q)\to \neg\neg p)\leftrightarrow \neg\neg p & (\text{Q}3)\\
  9^{\circ} & \ p\leftrightarrow \neg\neg p & 7^{\circ}, 8^{\circ}, (3)
\end{flalign*}
So $\vdash p\leftrightarrow \neg\neg p$.

\item We have
\begin{flalign*}
  1^{\circ} & \ (\neg p\to q)\leftrightarrow (\neg q\to \neg\neg p) & (\text{Q}1)\\
  2^{\circ} & \ \neg\neg p\leftrightarrow p & (10)\\
  3^{\circ} & \ (\neg p\to q)\leftrightarrow (\neg q\to p) & 1^{\circ}, 2^{\circ}, (6)
\end{flalign*}
So $\vdash (\neg p\to q)\leftrightarrow (\neg q\to p)$.

\item We have
\begin{flalign*}
  1^{\circ} & \ ((1\to 1)\to (\neg p)^{+})\leftrightarrow((\neg p\to 1)\to 1) & (\text{Q}11)\\
  2^{\circ} & \ ((\neg p\to 1)\to 1)\leftrightarrow \neg(1\to (\neg p\to 1)) & (\text{Q}5)\\
  3^{\circ} & \ ((1\to 1)\to (\neg p)^{+})\leftrightarrow \neg(1\to (\neg p\to 1)) & 1^{\circ}, 2^{\circ}, (3)\\
  4^{\circ} & \ 1\leftrightarrow \neg\neg 1 & (10)\\
  5^{\circ} & \ ((1\to 1)\to (\neg p)^{+})\leftrightarrow \neg(1\to (\neg p\to \neg\neg 1)) & 3^{\circ}, 4^{\circ}, (6)\\
  6^{\circ} & \ (\neg p\to \neg \neg 1)\leftrightarrow \neg(p\to \neg 1) & (4)\\
  7^{\circ} & \ ((1\to 1)\to (\neg p)^{+})\leftrightarrow \neg(1\to \neg(p\to \neg 1)) & 5^{\circ}, 6^{\circ}, (6)\\
  8^{\circ} & \ ((1\to 1)\to (\neg p)^{+})\leftrightarrow \neg(\neg \neg 1\to \neg(p\to \neg 1)) & 7^{\circ}, 4^{\circ}, (6)\\
  9^{\circ} & \ (\neg \neg 1\to \neg(p\to \neg 1))\leftrightarrow ((p\to \neg 1)\to \neg 1) & (\text{Q}1)\\
  10^{\circ} & \ ((1\to 1)\to (\neg p)^{+})\leftrightarrow \neg((p\to \neg 1)\to \neg 1) & 8^{\circ}, 9^{\circ}, (6)\\
  11^{\circ} & \ ((p\to \neg 1)\to \neg 1)\leftrightarrow ((1\to 1)\to p^{-}) & (\text{Q}11)\\
  12^{\circ} & \ ((1\to 1)\to (\neg p)^{+})\leftrightarrow \neg ((1\to 1)\to p^-)& 10^{\circ}, 11^{\circ}, (6)\\
  13^{\circ} & \ \neg((1\to 1)\to p^{-})\leftrightarrow (\neg(1\to 1)\to \neg p^-)& (4)\\
  14^{\circ} & \ ((1\to 1)\to (\neg p)^{+})\leftrightarrow (\neg(1\to 1)\to \neg p^-)& 12^{\circ}, 13^{\circ}, (3)\\
  15^{\circ} & \ \neg (1\to 1)\leftrightarrow (1\to 1)& (8)\\
  16^{\circ} & \ ((1\to 1)\to (\neg p)^{+})\leftrightarrow ((1\to 1)\to \neg p^-)& 14^{\circ}, 15^{\circ}, (6)\\
  17^{\circ} & \ (\neg p)^{+}\leftrightarrow ((1\to 1)\to (\neg p)^{+})& (\text{Q}3)\\
  18^{\circ} & \ (\neg p)^{+}\leftrightarrow ((1\to 1)\to \neg p^-)& 17^{\circ}, 16^{\circ}, (3)\\
  19^{\circ} & \ ((1\to 1)\to \neg p^-)\leftrightarrow \neg p^-& (\text{Q}3)\\
  20^{\circ} & \ (\neg p)^{+}\leftrightarrow \neg p^-& 18^{\circ}, 19^{\circ}, (3)
\end{flalign*}
So $\vdash (\neg p)^{+}\leftrightarrow \neg p^-$. From the former result, we have $\vdash (\neg \neg p)^{+}\leftrightarrow \neg (\neg p)^-$, then
\begin{flalign*}
  1^{\circ} & \ (\neg \neg p)^{+}\leftrightarrow \neg(\neg p)^- & (\text{Hypothesis})\\
  2^{\circ} & \ \neg \neg p \leftrightarrow p& (10)\\
  3^{\circ} & \ p^{+}\leftrightarrow \neg(\neg p)^- & 1^{\circ}, 2^{\circ}, (6)\\
  4^{\circ} & \ \neg p^{+}\leftrightarrow \neg \neg(\neg p)^- & 3^{\circ}, (1)\\
  5^{\circ} & \ \neg \neg(\neg p)^-\leftrightarrow (\neg p)^- & (10)\\
  6^{\circ} & \ \neg p^{+}\leftrightarrow (\neg p)^- & 4^{\circ}, 5^{\circ}, (6)
\end{flalign*}
So we have $\vdash \neg p^{+}\leftrightarrow (\neg p)^-$.
\item If $\vdash p\leftrightarrow q$, then we have
\begin{flalign*}
  1^{\circ} & \ p \leftrightarrow q & (\text{Hypothesis})\\
  2^{\circ} & \ 1\to 1  & (5)\\
  3^{\circ} & \ (p\to 1)\leftrightarrow (q\to 1) & 1^{\circ}, 2^{\circ}, (2)\\
  4^{\circ} & \ ((p\to 1)\to 1)\leftrightarrow ((q\to 1)\to 1)& 3^{\circ}, 2^{\circ}, (2)\\
  5^{\circ} & \ ((1\to 1)\to p^{+})\leftrightarrow ((p\to 1)\to 1)& (\text{Q}11)\\
  6^{\circ} & \ ((1\to 1)\to p^{+})\leftrightarrow ((q\to 1)\to 1)& 5^{\circ}, 4^{\circ}, (3)\\
  7^{\circ} & \ ((q\to 1)\to 1)\leftrightarrow ((1\to 1)\to q^{+})& (\text{Q}11)\\
  8^{\circ} & \ ((1\to 1)\to p^{+})\leftrightarrow ((1\to 1)\to q^{+})& 6^{\circ}, 7^{\circ}, (3)\\
  9^{\circ} & \ p^{+}\leftrightarrow ((1\to 1)\to p^{+})& (\text{Q}3)\\
  10^{\circ} & \ p^{+}\leftrightarrow ((1\to 1)\to q^{+})& 9^{\circ}, 8^{\circ}, (3)\\
  11^{\circ} & \ ((1\to 1)\to q^{+})\leftrightarrow q^{+}& (\text{Q}3)\\
  12^{\circ} & \ p^{+}\leftrightarrow q^{+}& 10^{\circ}, 11^{\circ}, (3)
\end{flalign*}
So $\vdash p^{+}\leftrightarrow q^{+}$. Moreover, by (1) and the former, we have $\vdash (\neg p)^{+}\leftrightarrow (\neg q)^{+}$, then
\begin{flalign*}
  1^{\circ} & \ (\neg p)^{+}\leftrightarrow (\neg q)^{+} & (\text{Hypothesis})\\
  2^{\circ} & \ \neg p^{-}\leftrightarrow (\neg p)^{+} & (12)\\
  3^{\circ} & \ \neg p^{-}\leftrightarrow (\neg q)^{+} & 2^{\circ}, 1^{\circ}, (3)\\
  4^{\circ} & \ (\neg q)^{+}\leftrightarrow \neg q^{-}& (12)\\
  5^{\circ} & \ \neg p^{-}\leftrightarrow\neg q^{-} & 3^{\circ}, 4^{\circ}, (3)\\
  6^{\circ} & \ \neg\neg p^{-}\leftrightarrow \neg\neg q^{-} & 5^{\circ}, (1)\\
  7^{\circ} & \ \neg\neg p^{-}\leftrightarrow p^{-} & (10)\\
  8^{\circ} & \ p^{-}\leftrightarrow \neg\neg q^{-}& 6^{\circ}, 7^{\circ}, (6)\\
  9^{\circ} & \ \neg\neg q^{-}\leftrightarrow q^{-} & (10)\\
  10^{\circ} & \ p^{-}\leftrightarrow  q^{-}& 8^{\circ}, 9^{\circ}, (6)\\
\end{flalign*}
So we have $\vdash p^{-}\leftrightarrow q^{-}$.
\end{enumerate}
\end{proof}

\begin{prop}\label{QL2}
Let $F(V)$ be the formulas set of \emph{q$\L^{*}$}.
For any $p, q\in F(V)$, we define a relation $\sim$ by
\[p \sim q \ \text{iff}\ \vdash p \leftrightarrow q.\]
Then the relation $\sim$ is a congruence with respect to the propositional connectives $\to, \neg, ^{+}$ and $^{-}$.
\end{prop}

\begin{proof}
Since $\vdash p \leftrightarrow p$ by Proposition 4.1(5), we have that the relation $\sim$ is reflexivity. It is obvious that the relation $\sim$ is symmetric. If $p\sim q$ and $q\sim r$, then $\vdash p \leftrightarrow q$ and $\vdash q \leftrightarrow r$,  we have $\vdash p \leftrightarrow r$ by Proposition \ref{QL1}(3) and then $p \sim r$. Hence the relation $\sim$ is an equivalence relation. If $p\sim q$, then $\vdash p \leftrightarrow q$,  we have $\vdash \neg p \leftrightarrow \neg q$ by Proposition \ref{QL1}(1), $\vdash p^+\leftrightarrow q^+$ and $\vdash p^-\leftrightarrow q^-$ by Proposition \ref{QL1}(13), so $\neg p \sim \neg q$, $p^+\sim q^+$ and $p^-\sim q^-$. If $p\sim q$ and $r\sim t$, then $\vdash p\leftrightarrow q$ and $\vdash r \leftrightarrow t$, we have $\vdash (p\to r) \leftrightarrow (q\to t)$ by Proposition \ref{QL1}(2), so $(p\to r) \sim (q\to t)$.  Hence the relation $\sim$ is a congruence with respect to the propositional connectives $\to, \neg, ^{+}$ and $^{-}$.
\end{proof}

Let $F(V)$ be a formulas set of q$\L^{*}$. For any $p\in F(V)$, the equivalence class of $p$ with respect to $\sim$ is denoted by $[p]_\sim =\{q\in F(V): p\sim q\}$ and all the equivalence classes of formulas in $F(V)$ is denoted by $[F(V)]_\sim$, i.e., $[F(V)]_\sim=\{[p]_\sim : p\in F(V)\}$. For any $[p]_\sim, [q]_\sim\in [F(V)]_\sim$, we can define $\neg[p]_\sim=[\neg p]_\sim$, $[p]_\sim \to [q]_\sim =[p\to q]_\sim$, $([p]_\sim)^{+}=[p^+]_\sim$ and $([p]_\sim)^{-}=[p^-]_\sim$.

\begin{prop}\label{QL3}
Let $F(V)$ be the formulas set of \emph{q$\L^{*}$}. Then $[\mathbf{F}(\mathbf{V})]_\sim=\langle [F(V)]_\sim; \to, \neg, ^{+}, ^{-}, [1]_\sim \rangle$ is a quasi-Wajsberg* algebra.
\end{prop}

\begin{proof}
Follows from the axioms schemas of q$\L^{*}$ and Proposition \ref{QL2}, we have that $[\mathbf{F}(\mathbf{V})]_\sim=\langle [F(V)]_\sim; \to, \neg, ^{+}, ^{-}, [1]_\sim \rangle$ is a quasi-Wajsberg* algebra.
\end{proof}

Moreover, we can define $-([p]_\sim)=\neg[ p]_\sim$, $[p]_\sim \oplus [q]_\sim =\neg[p]_\sim\to [q]_\sim$ and $[0]_\sim=[p\to p]_\sim$, then the following results hold.

\begin{coro}\label{cor1}
Let $F(V)$ be the formulas set of \emph{q$\L^{*}$}. Then $\langle [F(V)]_\sim;\oplus, -, ^{+}, ^{-}$, $[0]_\sim, [1]_\sim \rangle$ is a quasi-MV* algebra.
\end{coro}
\begin{proof}
Using Proposition \ref{p3.3} and Proposition \ref{QL3}, we have that $\langle [F(V)]_\sim;\oplus, -$, $^{+}, ^{-},$ $ [0]_\sim, [1]_\sim \rangle$ is a quasi-MV* algebra.
\end{proof}

\begin{prop}
Let $F(V)$ be the formulas set of \emph{q$\L^{*}$}. Then $\langle [F(V)]_\sim;\oplus, -, [0]_\sim, [1]_\sim \rangle$ is an MV*-algebra.
\end{prop}
\begin{proof}
For any $[q]_\sim\in [F(V)]_\sim$, we have $[q]_\sim\oplus [0]_\sim=\neg[q]_\sim\to [p\to p]_\sim=[\neg q\to (p\to p)]_\sim$. Now,
\begin{flalign*}
  1^{\circ} & \ (\neg q\to (p\to p))\leftrightarrow (\neg(p\to p) \to q) & \text{Proposition}\ \ref{QL1}(11)\\
  2^{\circ} & \ \neg(p\to p)\leftrightarrow(p\to p)& \text{Proposition}\ \ref{QL1}(8)\\
  3^{\circ} & \ (\neg q\to (p\to p))\leftrightarrow ((p\to p) \to q)& 1^{\circ}, 2^{\circ}, \text{Proposition}\ \ref{QL1}(6)\\
  4^{\circ} & \ ((p\to p) \to q)\leftrightarrow q& (\text{Q}3)\\
  5^{\circ} & \ (\neg q\to (p\to p))\leftrightarrow q& 3^{\circ}, 4^{\circ}, \text{Proposition}\ \ref{QL1}(3)
\end{flalign*}
so we have $[q]_\sim\oplus [0]_\sim=[\neg q\to (p\to p)]_\sim=[q]_\sim$. Hence $\langle [F(V)]_\sim;\oplus, -, [0]_\sim, [1]_\sim \rangle$ is an MV*-algebra by Corollary 4.1.
\end{proof}


In \cite{lsm2}, authors introduced the filter of MV*-algebra.
A subset $F$ is called a \emph{filter} of  MV*-algebra $\textbf{M}=\langle M;\oplus,^{+},^{-},0,1\rangle$, if the following conditions are satisfied:

(F1) $M^+=\{x^+: x\in M\}\subseteq F$,

(F2) If $x\in F$ and $y\ominus x\in F$, then $y\in F$,

(F3) If $x\oplus y\in F$ and $t\in A$, then $(x\oplus t)\oplus (y\ominus t)\in F$.

\begin{prop}
Let $\Theta$ be a \emph{q$\L^{*}$}-theory of \emph{q$\L^{*}$}. Then $[\Theta]_\sim=\{[p]_\sim\in [F(V)]_\sim: p\in \Theta\}$ is a filter of $\langle [F(V)]_\sim;\oplus, -,  [0]_\sim, [1]_\sim \rangle$.
\end{prop}

\begin{proof}
Since $\Theta$ is a q$\L^{*}$-\emph{theory} of q$\L^{*}$, we have that $p\in \Theta$ when $\Theta\vdash p$ for any $p\in F(V)$. Next we verify that $[\Theta]_\sim$ is a filter of $\langle [F(V)]_\sim;\oplus, -,  [0]_\sim, [1]_\sim \rangle$.

(F1) For any $([p]_\sim)^+\in ([F(V)]_\sim)^+$ where $p\in F(V)$, then
\begin{flalign*}
  1^{\circ} & \ ((p\to 1)\to 1)\to ((1\to 1)\to p^+)& (\text{Q}11)\\
  2^{\circ} & \ ((1\to 1) \to p^+)\to p^+ & (\text{Q}3)\\
  3^{\circ} & \ ((p\to 1)\to 1)\to p^+& 1^{\circ}, 2^{\circ}, \text{Proposition}\ \ref{QL1}(3)\\
  4^{\circ} & \ (p\to 1)\to 1 & (\text{Q}10)\\
  5^{\circ} & \ (1\to 1)\to p^+ & 4^{\circ}, 3^{\circ}, (\text{R}1)
\end{flalign*}
we have $\vdash (1\to 1)\to p^+$, so $(1\to 1)\to p^+\in \Theta$ and then $[(1\to 1)\to p^+]_\sim\in [\Theta]_\sim$. Since $[(1\to 1)\to p^+]_\sim=[1\to 1]_\sim\to [p^+]_\sim=[\mathbf{0}]_\sim\to [p^+]_\sim=\neg[\mathbf{0}]_\sim\to [p^+]_\sim=[\mathbf{0}]_\sim\oplus [p^+]_\sim=[p^+]_\sim$ by Proposition \ref{QL3}, Remark \ref{r1.1} and Corollary \ref{cor1}, we have $([p]_\sim)^+=[p^+]_\sim=[(1\to 1)\to p^+]_\sim\in [\Theta]_\sim$. So $([F(V)]_\sim)^+\subseteq [\Theta]_\sim$.

(F2) Assume that $[p]_\sim \in [\Theta]_\sim$ and $[q]_\sim\ominus [p]_\sim \in [\Theta]_\sim$, we have $p\in \Theta$, so $\Theta\vdash p$. From (Q1), we have $\vdash(p\to q)\leftrightarrow (\neg q\to \neg p)$, so $[p\to q]_\sim=[\neg q\to \neg p]_\sim$. Since $[\neg q\to \neg p]_\sim=\neg[q]_\sim\to \neg[p]_\sim=[q]_\sim\oplus (-[p]_\sim)=[q]_\sim\ominus [p]_\sim$, we have $[p\to q]_\sim\in [\Theta]_\sim$, which implies that $p\to q\in \Theta$, and then we have $\Theta\vdash p\to q$. Notice that $\Theta$ is closed under rules of deduction, we have $\Theta\vdash (1\to 1)\to q$ by (R1), it turns out that $(1\to 1)\to q\in \Theta$, so $[q]_\sim=[(1\to 1)\to q]_\sim\in [\Theta]_\sim$.

(F3) Assume that $[p]_\sim \oplus [q]_\sim\in [\Theta]_\sim$ and $[r]_\sim\in [F(V)]_\sim$, we have $[\neg p\to q]_\sim=\neg[p]_\sim \to [q]_\sim=[p]_\sim \oplus [q]_\sim \in [\Theta]_\sim$, so $\Theta\vdash \neg p\to q$.
Using Proposition \ref{QL1}(5), we have $\vdash r\to r$. Since $\Theta^{\vdash}=\Theta$, we have $\Theta\vdash r\to r$. So $\Theta \vdash (q\to r)\to (\neg p\to r)$ by (R3) and then we have $[(q\to r)\to (\neg p\to r)]_\sim\in[\Theta]_\sim$. Because $[(q\to r)\to (\neg p\to r)]_\sim=([q]_\sim\to [r]_\sim)\to (\neg [p]_\sim\to [r]_\sim)=\neg\neg(\neg\neg[q]_\sim\to [r]_\sim)\to (\neg [p]_\sim\to [r]_\sim)=-(-[q]_\sim\oplus [r]_\sim)\oplus([p]_\sim\oplus [r]_\sim)=([q]_\sim\oplus(- [r]_\sim))\oplus([p]_\sim\oplus [r]_\sim)=([p]_\sim\oplus [r]_\sim)\oplus ([q]_\sim\ominus[r]_\sim)$ by Proposition \ref{QL3}, (QW*8), Corollary \ref{cor1}, (QMV*10) and (QMV*1), we have $([p]_\sim\oplus [r]_\sim)\oplus ([q]_\sim\ominus[r]_\sim)\in [\Theta]_\sim$.

Hence $[\Theta]_\sim$ is a filter of the MV*-algebra $\langle [F(V)]_\sim;\oplus, -, [0]_\sim, [1]_\sim \rangle$.
\end{proof}

\begin{rem}
Let $F(V)$ be the formulas set of q$\L^{*}$ and denote $\text{R}(F(V))=\{p\in F(V): (1\to 1)\to p\vdash p\}$. Then for any $q\in \text{R}\{F(V)\}$, we have $\vdash_{\footnotesize \text{q}\L^{*}} q$ iff $\vdash_{\footnotesize \L^{*}} q$.
\end{rem}

\begin{defn}
Let $F(V)$ be the set of formulas of q$\L^{*}$. Then the mapping
$v^{*}: F(V)\to [-1,1]\times [-1,1]$ which is defined by
$ p \mapsto v^{*}(p)=\langle a, b\rangle$, for any $ a,b\in [-1,1]$,
is called a q$\L^{*}$-\emph{valuation}(\emph{valuation}, for short), if the following conditions are satisfied for any $p,q\in F(V)$:

(1) $v^{*}(1)=\langle 1, 0\rangle$,

(2) $v^{*}(\neg p)=\langle -a,-b\rangle$,

(3) $v^{*}(p\to q)=\langle \min\{1,\max\{-1,c-a\}\},0\rangle$ where $v^{*}(q)=\langle c,d\rangle$,

(4) $v^{*}(p^{+})=\langle \max\{0,a\},b\rangle$ and $v^{*}(p^{-})=\langle \min\{0,a\},b\rangle$.
\end{defn}

\begin{defn}
Let $F(V)$ be the set of formulas of q$\L^{*}$ and $p\in F(V)$. If $v^{*}(p)\in [0,1]\times [0,1]$ for any valuation $v^{*}$, then $p$ is called a \emph{tautology} and denoted by $\models_{\tiny \text{q}\L^{*}} p$.
\end{defn}

\begin{prop} Let $F(V)$ be the set of formulas of \emph{q$\L^{*}$}.
If $\vdash p$, then \emph{$\models_{\tiny \text{q}\L^{*}} p$}.
\end{prop}
\begin{proof}
We can get that axioms (Q1)--(Q11) are tautologies. And except for (Q10), which assignment belongs to $[0,1]\times[0,1]$, the other axioms are assigned to $\langle0,0\rangle$. Next we verify that the deduction rules keep tautologies.

(R1) For any $p, q\in F(V)$, if $\models_{\tiny \text{q}\L^{*}} p$ and $\models_{\tiny \text{q}\L^{*}} p\to q$, then for any valuation $v^{*}$, we have $v^{*}(p)=\langle a,b\rangle\in [0,1]\times[0,1]$ and $v^{*}(q)=\langle c,d\rangle\in [-1,1]\times[-1,1]$. Since $\models_{\footnotesize \text{q}\L^{*}} p\to q$, we have $v^{*}(p\to q)=\langle a,b\rangle\to \langle c,d\rangle=\langle \min\{1,\max\{-1, c-a\}\}, 0\rangle=\langle c-a, 0\rangle\in [0,1]\times[0,1]$, so $c-a\in [0,1]$. Note that $a\in [0,1]$, we have $0\leq c$, so $v^{*}((1\to 1)\to q)=(v^{*}(1)\to v^{*}(1))\to v^{*}(q)=(\langle 1, 0\rangle\to \langle 1, 0\rangle)\to \langle c, d\rangle=\langle 0, 0\rangle\to \langle c, d\rangle=\langle c, 0\rangle\in [0,1]\times[0,1]$ and then $\models_{\tiny \text{q}\L^{*}} (1\to 1)\to q$.

(R2) If $\models_{\tiny \text{q}\L^{*}} ((1\to 1)\to (p\to q))$, then for any valuation $v^{*}$, we have that $v^{*}((1\to 1)\to (p\to q))\in [0,1]\times[0,1]$. Since $v^{*}(p\to q)=(\langle 1, 0\rangle\to \langle 1, 0\rangle)\to v^{*}(p\to q)=(v^{*}(1)\to v^{*}(1))\to v^{*}(p\to q)=v^{*}((1\to 1)\to (p\to q))$, we have $v^{*}(p\to q)\in [0,1]\times[0,1]$, so $\models_{\tiny \text{q}\L^{*}} p\to q$.

(R3) If $\models_{\tiny \text{q}\L^{*}} p\to q$ and $\models_{\tiny \text{q}\L^{*}} r\to t$, then for any valuation $v^{*}$, we have that $v^{*}(p\to q)=\langle \min\{1,\max\{-1, c-a\}\},0\rangle\in [0,1]\times[0,1]$ where $v^{*}(p)=\langle a,b\rangle$ and $v^{*}(q)=\langle c,d\rangle$,
$v^{*}(r\to t)=\langle\min\{1,\max\{-1, h-e\}\},0\}\rangle\in [0,1]\times[0,1]$ where $v^{*}(r)=\langle e,f\rangle$ and $v^{*}(t)=\langle h,j\rangle$, so $c-a\in [0,1]$ and $h-e\in[0,1]$.
We calculate that $v^{*}((q\to r)\to (p\to t))=\langle \min\{1,\max\{-1, h-e+c-a\},0\}\rangle$ and $0\leq h-e+c-a$,
so $v^{*}((q\to r)\to (p\to t))\in [0,1]\times [0,1]$ and then $\models_{\tiny \text{q}\L^{*}} ((q\to r)\to (p\to t))$.

\end{proof}

\section{Conclusions}\label{Con}
In this paper, we primarily investigate the logical system associated with quasi-MV* algebras. To achieve this, we introduce the definition of quasi-Wajsberg* algebras and investigate their related properties. Furthermore, we establish the term equivalence between quasi-Wajsberg* algebras and quasi-MV* algebras. Additionally, we construct the logical system associated with quasi-Wajsberg* algebras and prove its soundness. In future work, we aim to investigate the completeness of this logical system and explore further properties related to complex fuzzy logic.

\section{Declarations}
\begin{description}
\item[Funding] This work was supported by Shandong Provincial Natural Science Foundation, China
(No. ZR2020MA041), China Postdoctoral Science Foundation (No. 2017M622177) and Shandong Province Postdoctoral Innovation Projects of Special Funds (No. 201702005).
\item[Conflict interest] Author A declares that she has no conflict of interest. Author B declares that she has no conflict of interest.
\end{description}

\raggedright
\end{document}